\documentclass[10pt, twoside]{amsart}
\usepackage[latin1]{inputenc}
\usepackage[T1]{fontenc}
\usepackage{amsfonts}
\usepackage{amssymb}
\usepackage{amsthm}
\usepackage{latexsym}
\usepackage{mathrsfs}
\usepackage{amsmath}
\usepackage{verbatim}
\usepackage{enumerate}

\numberwithin{equation}{section}
\makeatletter
\newcommand\ackname{Acknowledgements}
\if@titlepage
  \newenvironment{acknowledgements}{%
      \titlepage
      \null\vfil
      \@beginparpenalty\@lowpenalty
      \begin{center}%
        \bfseries \ackname
        \@endparpenalty\@M
      \end{center}}%
     {\par\vfil\null\endtitlepage}
\else
  
\fi
\makeatother

\newcommand{\bfi}{\begin{fig}}
\newcommand{\efi}{\end{fig}}
\newcommand{\btab}{\begin{tab}}
\newcommand{\etab}{\end{tab}}
\newcommand{\barr}{\begin{array}}
\newcommand{\earr}{\end{array}}
\newcommand{\beqq}{\begin{equation}}
\newcommand{\eeqq}{\end{equation}}
\newcommand{\beao}{\begin{align*}}
\newcommand{\eeao}{\end{align*}\noindent}
\newcommand{\beam}{\begin{eqnarray}}
\newcommand{\eeam}{\end{eqnarray}\noindent}
\newcommand{\bdis}{\begin{displaymath}}
\newcommand{\edis}{\end{displaymath}\noindent}

\newcommand{\bbn}{\mathbb{N}}

\newcommand{\bbr}{\mathbb{R}}

\newcommand{\calS}{\mathcal{S}}
\newcommand{\bfS}{\mathbf{S}}

\newcommand{\eps}{{\epsilon}}

\newcommand{\si}{{\sigma}}
\newcommand{\Om}{{\Omega}}

\newtheorem{Satz}{Satz}[section]
\newtheorem{Theorem}[Satz]{Theorem}
\newtheorem{Korollar}[Satz]{Corollary}

\newtheorem{Proposition}[Satz]{Proposition}

\newtheorem{Lemma}[Satz]{Lemma}

\theoremstyle{definition}
\newtheorem{Definition}[Satz]{Definition}
\newtheorem{Annahme}[Satz]{Assumption}

\newtheorem{Bemerkung}[Satz]{Remark}

\theoremstyle{remark}

\begin{document}

\pagenumbering{arabic}
\setcounter{page}{1}

\title{Volume preserving curvature flows in Lorentzian manifolds}
\author{Matthias Makowski}
\date{April 12, 2011}
\address{Matthias Makowski, Universit\"at Konstanz, 78457 Konstanz, Germany}
\def\fuaddress{@uni-konstanz.de}
\email{Matthias.Makowski\fuaddress}
\address{http://www.math.uni-konstanz.de/\~{}makowski/}
\begin{abstract}
Let $N$ be a $(n+1)$-dimensional globally hyperbolic Lorentzian manifold with a compact Cauchy hypersurface $\mathcal{S}_0$ and $F$ a curvature function, either the mean curvature $H$, the root of the second symmetric polynomial $\si_2 = \sqrt{H_2}$ or a curvature function of class $(K^*)$. We consider curvature flows with curvature function $F$ and a volume preserving term and prove long time existence of the flow and exponential convergence of the corresponding graphs in the $C^\infty$-topology to a hypersurface of constant $F$-curvature, provided there are barriers. Furthermore we examine stability properties and foliations of constant $F$-curvature hypersurfaces.
\end{abstract}
\maketitle
\tableofcontents

\nopagebreak

\section{Introduction}
\label{Introduction}

We show the long time existence and convergence to a constant $F$-hypersurface of the following curvature flow in a globally hyperbolic Lorentzian manifold with compact Cauchy hypersurface under suitable assumptions:
\begin{equation}
\label{floweq}
\begin{split}
	\dot{x} &= (\Phi(F) - f)\, \nu,\\
	x(0) &= x_0,
\end{split}
\end{equation}
where $x_0$ is the embedding of an initial, compact, connected, spacelike hypersurface $M_0$ of class $C^{m+2,\alpha}$, $2\leq m \in \bbn$, $0< \alpha < 1$, $\nu$ is the corresponding past directed normal, $F$ is a curvature function of class $C^{m,\alpha}(\Gamma)$ evaluated at the principal curvatures of the flow hypersurfaces $M(t)$, $x(t)$ denotes the embedding of $M(t)$, $\Phi$ is a smooth supplementary function satisfying $\Phi' > 0$, $\Phi'' \leq 0$, and $f$ is a volume preserving global term, $f = f_k$, see the definition below.

Furthermore the initial hypersurface should be \begin{em}admissible\end{em}, meaning that its principal curvatures belong to the defining cone $\Gamma$ of the curvature function $F$, which will be specified below.

Depending on which type of volume has to be preserved, we define the global term as in \cite{McCoyMixedAreaGen}:
\begin{equation}
\label{globTerm}
	f_k(t) = \frac{\int_{M_t}{H_k \Phi(F)\, \mathrm{d\mu_t}}}{\int_{M_t}{H_k\,\mathrm{d\mu_t}}}.
\end{equation}
Here $H_k$, $k = 0, ..., n$,  denotes the k-th elementary symmetric polynomial, where $H_0 = 1$. For an overview of the notation (especially concerning the curvature functions) we refer to section 2.

We assume that the ambient space $N$ is a $(n+1)$-dimensional smooth, connected, globally hyperbolic Lorentzian manifold with a compact, smooth, connected Cauchy hypersurface $\mathcal{S}_0$, and $N$ is covered by a future directed Gaussian coordinate system $(x^\alpha)$, such that the metric $(\bar{g}_{\alpha\beta})$ can be expressed in the form
\begin{equation}
\label{GKS}
	d\bar{s}^2 = e^{2\psi(x^0, x)}\{ - (dx^0)^2 + \sigma_{ij}(x^0, x)\, dx^idx^j\},
\end{equation}
where $x^0$ is the time function defined on an interval $I = (a,b)$, we suppose without loss of generality $0 \in I$ and $(x^i)$ are local coordinates for the Cauchy hypersurface $\mathcal{S}_0$. The coordinates can be chosen such that
\begin{equation}
	\mathcal{S}_0 = \{x^0 = 0\}.
\end{equation}
The existence of a smooth, proper function $f:N \rightarrow \bbr$ with non-vanishing timelike gradient in a merely connected, smooth Lorentzian manifold $N$ already assures the existence of such a special coordinate system, see \cite[Theorem 1.4.2]{GerhCP}, implying that $N$ is globally hyperbolic with a compact Cauchy hypersurface. Alternatively one can deduce the existence of the special coordinate system in smooth, globally hyperbolic Lorentzian manifolds with compact Cauchy hypersurface from \cite[Theorem 1.1]{Split2} and \cite[Lemma 2.2]{Split1}.

We need one further assumption on the ambient manifold, namely we consider curvature flows in cosmological spacetimes, a terminology due to Bartnik, meaning a Lorentzian manifold with the above properties, which furthermore satisfies the timelike convergence condition, an assumption which is quite natural in the setting of general relativity as it corresponds to the strong energy condition (see for example \cite{HawEll}). Hence for all $p\in N$ there holds
\begin{equation}
\label{TCC}
	\bar{R}_{\alpha\beta}V^\alpha V^\beta \geq 0 \quad\forall \text{ timelike } V \in T_pN.
\end{equation}
We only mention that for the proof of Theorem \ref{MainTheorem1} (with $F=H$) this condition could be relaxed to the case where the lower bound is $-\Lambda$ with a constant $\Lambda > 0$, where in this case one needs to assume that there holds $H > \sqrt{n\Lambda}$ on the initial hypersurface.

In the case of general curvature functions however we will need to assume that the timelike sectional curvatures of $N$ are non-positive, i.e. at points $p \in N$ there holds
\begin{equation}
\label{NPTSC}
	\bar{R}_{\alpha\beta\gamma\delta}V^\alpha W^\beta V^\gamma W^\delta \geq 0 \quad\forall \text{ timelike } V\in T_pN, \,\forall\, \text{ spacelike } W \in T_pN.
\end{equation}

The possible curvature functions are $F =H$, $F=\si_2$ or $F \in (K^*)$. For these we have to distinguish their cones of definition $\Gamma$ and the supplementary function $\Phi$:
\begin{itemize}
\item
Let $F = H$ and $k=0$, then let $\Phi(x) = x$ and $\Gamma = \bbr^n$. For $k=1$ let $\Gamma = \Gamma_1$ and $\Phi \in C^{m,\alpha}(\bbr_+)$ be an arbitrary function satisfying merely $\Phi' > 0$ and $\Phi'' \leq 0$. For example, one could consider the surface-area preserving inverse mean curvature flow, $\Phi(x) = -x^{-1}$. For higher $k$ the flow is not well defined, since convexity does not need to be preserved during the flow.
\item
For $F = \si_2 = H_2^{\frac{1}{2}}$ let $\Gamma = \Gamma_2$ and $\Phi(x) = x$ or $\Phi(x) = -x^{-1}$. Again, the flow is only well defined for $k \in \{0,1,2\}$ for the same reasons as above.
\item
Lastly, let $F \in (K^*)$ be a homogeneous function of degree 1 and of class $C^{m,\alpha}(\Gamma_+)$, then for $k \in \{0, \ldots, n\}$ we choose $\Phi(x) = \log(x)$ and $\Gamma = \Gamma_+$.
\end{itemize}

We denote by $(F, \Gamma, \Phi)$ one of the possible choices of curvature functions and their respective cones of definitions as well as supplementary functions stated above.
	
In order to be able to derive $C^0$-estimates we have to add an additional assumption, first we provide the necessary definition:
\begin{Definition}
\label{DefVolDec}
Let $F$ be a continuous curvature function defined on an open, convex, symmetric cone $\Gamma \subset \bbr^n$. Then we define (here we distinguish the cases considering the future or the past by brackets):

Let $c$ be a constant, then we say we have a future (past) curvature barrier for $(F, \Gamma, c)$ of class $C^{k, \beta}$, where $k\in \bbr$, $k\geq 2$, $0\leq \beta \leq 1$, if there exists a compact, connected, spacelike and admissible hypersurface $M$ of class $C^{k, \beta}$, satisfying
\begin{equation}
\label{Barrier}
	F_{|M} \geq (\leq)\, c.
\end{equation}
\end{Definition}

With the definition
\begin{equation}
	c_1 = \underset{M_0}{\min}\, F \quad \textnormal{ and } \quad c_2 = \underset{M_0}{\max}\, F,
\end{equation}
we can state the following assumption:
\begin{Annahme}
\label{MainAssumption}
	We have a future curvature barrier for $(F, \Gamma, c_2)$ of class $C^{2}$ and a past curvature barrier for $(F, \Gamma, c_1)$ of class $C^{2}$. If in the case $F=H$, $\Gamma = \bbr^n$, for $i=1$ or $i=2$ there holds $c_i = 0$, then we assume the corresponding barrier to be strict.
	
	If the curvature function is not the mean curvature, we assume the existence of a strictly convex function $\chi_\Om \in C^2(\bar{\Om})$, where $\Om \subset N$ is the region between the barriers. For geometric conditions implying the existence of such a function see Lemma \ref{ExistenceStrictlyConvex}.
\end{Annahme}

Now we state the theorem:
\begin{Theorem}
\label{MainTheorem1}
	Let $N$, $M_0$ and $(F, \Gamma, \Phi)$ be as above, $m \geq 2$, $0< \alpha < 1$, and suppose there holds assumption \ref{MainAssumption}. Then the flow \eqref{floweq} with $f= f_k$ has a unique solution existing for all times $0\leq t < \infty$, such that for fixed time $M_t \in C^{m+2,\alpha}$ and the $M_t's$ considered as graphs $u(t, \cdot)$ converge exponentially in $C^{m+2}$ to a compact, connected, spacelike hypersurface of class $C^{m+2,\alpha}$, which is a stable solution of the equation
	\begin{equation}
		F = c_0,
	\end{equation}
	where $c_0 = \lim\limits_{t\rightarrow \infty}{\Phi^{-1}(f_k)}$.
	
If $M_0$ and $F$ are smooth, then the convergence of the graphs is exponential in the $C^\infty$-topology.

For $k=0$ the enclosed volume, for $k=1$ the volume of the hypersurfaces and if the ambient space has constant curvature $K_N = 0$, then for $1 <k \leq n$ the mixed volume $V_{n+1-k}$ is preserved.
\end{Theorem}

Finally, we want to name some of the works about volume preserving curvature flows in different ambient manifolds and discuss shortly the results obtained in this work.

Volume preserving curvature flows have been considered for various curvature functions in different settings. Roughly speaking, if one assumes a certain convexity assumption or pinching condition on the initial hypersurface and shows that this condition is preserved during the flow, then after proving a priori estimates the existence of the flow for all times $t \in [0,\infty)$ and the exponential convergence in the $C^\infty$-topology of the flow to a sphere or a geodesic sphere can be deduced.

In the case the ambient manifold is $\bbr^{n+1}$, volume preserving mean curvature flows have been previously considered by Gage for $n=1$ in \cite{Gage} and by Huisken for $n\geq 2$ in \cite{HuisVol}. McCoy considered mixed volume preserving mean curvature flows in $\bbr^{n+1}$ in \cite{McCoyMixedArea} and later on extended the results to very general curvature functions in \cite{McCoyMixedAreaGen}.

Recently Cabezas-Rivas and Miquel proved similar results for a volume preserving mean curvature flow in the hyperbolic space under the assumption of horosphere-convexity of the initial hypersurface, see \cite{CabMiqHyp}. Cabezas-Rivas and Sinestrari then considered the volume-preserving flow by powers of the elementary symmetric polynomials in the euclidean setting in \cite{RivSin} by assuming a pinching condition on the principal curvatures of the initial hypersurface.

However, to our knowledge the only result concerning volume preserving curvature flows in Lo\-rentzian manifolds can be found in the paper \cite{EckerHuisken} by Ecker and Huisken, where the volume preserving mean curvature flow has been considered. The method in the Lorentzian case differs substantially from the euclidean case. Neither convexity nor the pinching condition on the principal curvatures is preserved, but assuming \eqref{TCC} in the case of $F=H$ and \eqref{NPTSC} in the case of a general curvature function respectively, investigating the evolution equation for the curvature function one can see that the upper and lower bound of the curvature function is preserved during the flow, which is also valid if an arbitrary, but bounded global term is considered. This result is the crucial part that enables one to prove $C^0$-estimates under the assumption of barriers. Now the $C^1$ and $C^2$-a priori estimates can be deduced by the same methods used in the case of a time-independent force-term and do not rely on the special choice of the global term. The higher order estimates can not be deduced directly from the results of Krylov-Safonov in view of the global term (which is merely bounded at this moment), instead we use a method already employed in the papers \cite{McCoyMixedAreaGen} and \cite{RivSin}. Then again the evolution equation for the curvature function is the starting point to conclude the exponential convergence to a hypersurface of constant $F$-curvature.

From the above remark about the dependence of the proofs on the global term $f$ one can conclude, that, as far as long time existence is concerned, a far wider class of global terms can be considered than the ones used throughout the paper. In particular one can look as well at curvature flows that preserve volumes with different densities and obtain the same results stated above, as they neither disturb the boundedness of the curvature function nor the analysis carried out to achieve convergence.

It is also possible to prove the foliation of a future end of $N$ by CMC-\-hyper\-surfaces by a similar method as in \cite{GerhCP} by using the volume preserving curvature flow. However, since the proof is more complicated than the proof by using the mean curvature flow without a global term, we omit the proof of this result. Instead, we show in section 10 that a region enclosed by barriers for the $F$-curvature can be foliated by hypersurfaces of constant $F$-curvature. Furthermore we show that each CFC-surface in the interior of this region can be obtained as the limit hypersurface of a nontrivial curvature flow which preserves the volume respectively the area. 

{\bf Acknowledgement:} This work is part of the doctoral dissertation of the author at the University of Heidelberg. The author wishes to thank Prof. Dr. C. Gerhardt for the introduction to the subject of geometric analysis and for the most important parts of his mathematical education in general.

\newpage

\section{Notation and Definitions}
\label{Notation}

The main objective of this section is to formulate the governing equations of a hypersurface in a Lorentzian $(n+1)$-dimensional manifold $N$ and to provide the definitions of the classes of curvature as well as some well-known properties of certain curvature functions which will be used throughout this paper. Note that the main differences of hypersurfaces in Lorentzian manifolds compared to hypersurfaces in Riemannian manifolds arise from the sign change in the Gauß formula and hence the Gauß equation. For more detailed definitions about curvature functions, we refer the reader to \cite[Chapter 2.1, 2.2]{GerhCP} and for an account of the differential geometry to \cite[Chapter 11, 12]{GerhAna} and especially Chapter 12.5 therein with respect to Gaussian coordinate systems and Lorentzian manifolds.

Throughout this section $N$ will be assumed to be a $(n+1)$-dimensional Lorentzian manifold and, unless stated otherwise, the summation convention is used throughout the paper.

We will denote geometric quantities in the ambient space $N$ by greek indices with range from $0$ to $n$ and usually with a bar on top of them, for example the metric and the Riemannian curvature tensor in the ambient space will be denoted by $(\bar{g}_{\alpha\beta})$ and $(\bar{R}_{\alpha\beta\gamma\delta})$ respectively, etc., and geometric quantities of a spacelike hypersurface $M$ by latin indices ranging from $1$ to $n$, i.e. the induced metric and the Riemannian curvature tensor on M are denoted by $(g_{ij})$ and $(R_{ijkl})$ respectively. Generic coordinate systems in $N$ and $M$ will be denoted by $(x^\alpha)$ and $(\xi^i)$ respectively. Ordinary partial differentiation will be denoted by a comma whereas covariant differentiation will be indicated by indices or in case of possible ambiguity they will be preceded by a semicolon, i.e. for a function $u$ in $N$, $(u_\alpha)$ denotes the gradient and $(u_{\alpha\beta})$ the Hessian, but e.g. the covariant derivative of the curvature tensor will be denoted by $(\bar{R}_{\alpha\beta\gamma\delta;\epsilon})$. We also point out that (with obvious generalizations to other quantities)
\begin{equation}
	\bar{R}_{\alpha\beta\gamma\delta;i} = \bar{R}_{\alpha\beta\gamma\delta;\epsilon} x_i^\epsilon,
\end{equation}
where $x$ denotes the embedding of $M$ in $N$ in local coordinates $(x^\alpha)$ and $(\xi^i)$.

The induced metric of the hypersurface will be denoted by $g_{ij}$, i.e.
\begin{equation}
	g_{ij} = \langle x_i, x_j \rangle \equiv \bar{g}_{\alpha \beta} x_i^\alpha x_j^\beta,
\end{equation}
the second fundamental form will be denoted by $(h_{ij})$ and the normal by $\nu$, which is a \begin{em}timelike\end{em} vector, i.e. for $p \in M$ there holds
\begin{equation}
	\nu(p) \in C_p := \{\xi \in T_p^{1,0}(N): \langle \xi, \xi \rangle < 0\},
\end{equation}
where $T_p^{k, l}(N)$ denotes the k-times contravariant and l-times covariant tensors and we note that the light cone $C_p$ consists of two connected components, $C_p^+$ and $C_p^-$, which we call \begin{em} future directed \end{em} and \begin{em} past directed\end{em} respectively.

The geometric quantities of the spacelike hypersurface $M$ are connected through the \begin{em}Gauß formula\end{em}, which can be considered as the definition of the second fundamental form,
\begin{equation}
	x_{ij} = h_{ij}\nu,
\end{equation}
where we are free to choose the future or the past directed normal, but we stipulate that we always use the past directed normal.

Note that here and in the sequel a covariant derivative is always a full tensor, i.e.
\begin{equation}
	x_{ij}^\alpha = x_{,ij}^\alpha -\Gamma_{ij}^k x_k^\alpha + \bar{\Gamma}_{\beta\gamma}^\alpha x_i^\beta x_j^\gamma,
\end{equation}
where $\bar{\Gamma}^\alpha_{\beta \gamma}$ and $\Gamma^k_{ij}$ denote the Christoffel-symbols of the ambient space and hypersurface respectively.

The second equation is the \begin{em}Weingarten equation\end{em}:
\begin{equation}
	\nu_i = h_i^k x_k = g^{kj} h_{ij} x_k.
\end{equation}

Finally, we have the \begin{em}Codazzi equation\end{em}
\begin{equation}
	h_{ij;k} = h_{ik;j} + \bar{R}_{\alpha\beta\gamma\delta}\nu^\alpha x_i^\beta x_j^\gamma x_k^\delta,
\end{equation}
as well as the \begin{em}Gauß equation\end{em}
\begin{equation}
	R_{ijkl} = -\{h_{ik}h_{jl} - h_{il}h_{jk}\} + \bar{R}_{\alpha\beta\gamma\delta}x^\alpha_i x_j^\beta x_k^\gamma x_l^\delta.
\end{equation}
Note that in the last equation the sign change comes into play. 

Now we want to define the different classes of curvature functions, first we provide the definition of such functions and mention some identifications, which will be used in the sequel without explicitly stating them again.
\begin{Definition}
	Let $\Gamma \subset \bbr^n$ be an open, convex, symmetric cone, i.e.
	\begin{equation}
		(\kappa_i) \in \Gamma \Longrightarrow (\kappa_{\pi i}) \in \Gamma \quad \forall \, \pi \in \mathcal{P}_n,
	\end{equation}
	where $\mathcal{P}_n$ is the set of all permutations of order $n$. Let $f \in C^{m, \alpha}(\Gamma)$, $m \in \bbn$, $0\leq \alpha \leq 1$, be \textit{symmetric}, i.e.,
	\begin{equation}
		f(\kappa_i) = f(\kappa_{\pi i}) \quad \forall \, \pi \in \mathcal{P}_n.
	\end{equation}
	Then $f$ is said to be a \textit{curvature function} of class $C^{m,\alpha}$. For simplicity we will also refer to the pair $(f, \Gamma)$ as a curvature function.
\end{Definition}
Now denote by $\mathbf{S}$ the symmetric endomorphisms of $\bbr^n$ and by $\mathbf{S}_\Gamma$ the symmetric endomorphisms with eigenvalues belonging to $\Gamma$, an open subset of $\mathbf{S}$. Then we can define a mapping
\begin{equation}
\begin{split}
	F: &\mathbf{S} \rightarrow \bbr,\\
	&A\mapsto f(\kappa_i),
\end{split}
\end{equation}
where the $\kappa_i$ denote the eigenvalues of $A$. For the relation between these different notions, especially the differentiability properties and the relation between their derivatives, see \cite[Chapter 2.1]{GerhCP}. Since the differentiability properties are the same for $f$ as for $F$ in our setting, see \cite[Theorem 2.1.20]{GerhCP}, we do not distinguish between these notions and write always $F$ for the curvature function. Hence at a point $x$ of a hypersurface we can consider a curvature function $F$ as a function defined on a cone $\Gamma \subset \bbr^n$, $F = F(\kappa_i)$ for $(\kappa_i) \in \Gamma$ (representing the principal curvatures at the point $x$ of the hypersurface), as a function depending on $(h_i^j)$, $F = F(h_i^j)$ or even as a function depending on $(h_{ij})$ and $(g_{ij})$, $F = F(h_{ij}, g_{ij})$. However, we distinguish between the derivatives with respect to $\Gamma$ or $\mathbf{S}$. We summarize briefly our notation and important properties:

	For a (sufficiently smooth) curvature function $F$ we denote by $F^{ij} = \frac{\partial F}{\partial h_{ij}}$, a contravariant tensor of order 2, and $F^j_i = \frac{\partial F}{\partial h_j^i}$, a mixed tensor, contravariant with respect to the index $j$ and covariant with respect to $i$. We also distinguish the partial derivative $F_{,i} = \frac{\partial F}{\partial \kappa_i}$ and the covariant derivative $F_{;i} = F^{kl}h_{kl;i}$. Furthermore $F^{ij}$ is diagonal if $h_{ij}$ is diagonal and in such a coordinate system there holds $F^{ii} = \frac{\partial F}{\partial \kappa_i}$. For a relation between the second derivatives see \cite[Lemma 2.1.14]{GerhCP}. Finally, if $F \in C^2(\Gamma)$ is concave, then $F$ is also concave as a curvature function depending on $(h_{ij})$. 
With these definitions we can turn to special classes of curvature functions. 

But first we remind the definition of an admissible hypersurface:
\begin{Definition}
A spacelike, orientable hypersurface $M$ of class $C^2$ in a Lo\-rentzian manifold $N$ is said to be \begin{em}admissible\end{em} with respect to a continuous curvature function $(F, \Gamma)$, if its principal curvatures with respect to the past directed normal lie in $\Gamma$.
\end{Definition}\begin{Definition}
\label{curvClass}
We distinguish three classes of curvature functions:
\begin{itemize}
	\item[(i)] A symmetric curvature function $F \in C^{2,\alpha}(\Gamma_+)\cap C^0(\bar{\Gamma}_+)$, where $\Gamma_+ := \{(\kappa_i) \in \bbr^n: \kappa_i > 0, 1 \leq i \leq n\}$, positively homogeneous of degree $d_0 > 0$, is said to be of class $(K)$, if 
		it is strictly monotone, i.e.
		\begin{equation} F_{,i} = \frac{\partial F}{\partial \kappa^i} > 0 \quad \text{in } \Gamma_+ \, \text{,} \end{equation}
		vanishes on the boundary of $\Gamma_+$ and fulfills the following inequality:
		\begin{equation}
			F^{ij,kl}\eta_{ij}\eta_{kl} \leq F^{-1}(F^{ij}\eta_{ij})^2 - F^{ik}\tilde{h}^{jl}\eta_{ij}\eta_{kl} \quad \forall\, \eta \in \mathbf{S},
		\end{equation}
		where $F$ is evaluated at $(h_{ij}) \in \bfS_{\Gamma_+}$ and $(\tilde{h}^{ij})$ is the inverse of $(h_{ij})$.
	\item[(ii)] A function $F \in (K)$ is said to be of class $(K^*)$ if there exists $0 < \epsilon_0 = \epsilon_0(F)$ such that
		\begin{equation}
			\epsilon_0 F H\leq F^{ij}h_{ik} h^k_j \quad \forall\, (h_{ij}) \in \mathbf{S}_{\Gamma_+},
		\end{equation}
		where $F$ is evaluated at $(h_{ij})$ and $H$ represents the mean curvature, i.e. the sum of the eigenvalues of $(h_{ij})$.
	\item[(iii)] A differentiable curvature function $F$ is said to be of class $(D)$, if for every admissible hypersurface $M$ the tensor $F^{ij}$, evaluated at $M$, is divergence free.
\end{itemize}
\end{Definition}
First, we define the most important curvature functions, the elementary symmetric polynomials
\begin{equation}
H_k(\lambda_1, \cdots, \lambda_n) = \sum_{i_1<\cdots <i_k} \lambda_{i_1}\cdots\lambda_{i_k},\quad \lambda = (\lambda_i) \in \bbr^n, \, 1\leq k \leq n,
\end{equation} 
and note that the $n$-th root of the gaussian curvature $\si_n = K^{\frac{1}{n}} = H_n^{\frac{1}{n}}$ is an example of a curvature function of class $(K^*)$. 

Further examples are given by noticing that if $F \in (K^*)$ then $F^a \in (K^*)$ for $a > 0$ and if furthermore $G \in (K)$ (where in this case it does not have to vanish on the boundary) then $FG \in (K^*)$. Possible choices of $G$ would be the inverses of the symmetric polynomials $\tilde{H}_k(\kappa_i) = \frac{1}{H_k(\kappa_i^{-1})}$, see \cite[Chapter 2.2]{GerhCP}. Secondly, we remark, that a curvature function of class $(K)$ and homogeneous of degree 1 is also concave, see \cite[Lemma 2.2.14]{GerhCP}.

We note some important properties of the elementary symmetric polynomials:
\begin{Lemma}
\label{symPol}
Let $1\leq k \leq n$ be fixed.
	\begin{itemize}
		\item[(i)] We define the convex cone 
			\begin{equation}
				\Gamma_k = \{(\kappa_i) \in \bbr^n: H_1(\kappa_i) > 0, H_2(\kappa_i) > 0, \ldots, H_k(\kappa_i) > 0 \}.
			\end{equation}
			Then $H_k$ is strictly monotone on $\Gamma_k$ and $\Gamma_k$ is exactly the connected component of
			\begin{equation}
				\{(\kappa_i) \in \bbr^n: H_k(\kappa_i) > 0\}
			\end{equation}
			containing the positive cone.
		\item[(ii)] The $k$-th roots $\sigma_k = H_k^{\frac{1}{k}}$ are concave on $\Gamma_k$.
		\item[(iii)] For $ 1 < s < t < n$ and $\tilde{\sigma}_k = \Bigl(\frac{H_k}{{n \choose k}}\Bigr)^{\frac{1}{k}}$ there holds
			\begin{equation}
				\tilde{\sigma}_n \leq \tilde{\sigma}_t \leq \tilde{\sigma}_s \leq \tilde{\sigma}_1,
			\end{equation}
			where the principal curvatures have to lie in $\Gamma_n \equiv \Gamma_+$ for the first, in $\Gamma_t$ for the second and in $\Gamma_s$ for the third inequality.
		\item[(iv)] For fixed $i$, no summation over $i$, there holds
			\begin{equation}
				H_k = \frac{\partial H_{k+1}}{\partial \kappa_i} + \kappa_i \frac{\partial H_k}{\partial \kappa_i}.
			\end{equation}
	\end{itemize}
\end{Lemma}
\begin{proof}
	The convexity of the cone $\Gamma_k$ and (i) follows from \cite[Section 2]{HuiskSinestr}, (ii) and (iii) from \cite{Lieberman}, Lemma 15.12 and Theorem 15.16, and (iv) follows directly from the definition of the $H_k$.
\end{proof}
A consequence of the preceding Lemma is the following
\begin{Lemma}
\label{divFree}
	Let $N$ be a semi-Riemannian space of constant curvature, then the symmetric polynomials $F= H_k$, $1\leq k \leq n$, are of class $(D)$. In case $k=2$ it suffices to assume that $N$ is an Einstein manifold.
\end{Lemma}
\begin{proof}
	The proof of the Lemma can be found in \cite[Lemma 5.8]{GerhSurvey}. The proof consists of induction on $k$ and (iv) of Lemma \ref{symPol}.
\end{proof}
Now we state a well-known inequality for general curvature functions:
\begin{Lemma}
\label{FHineq}
	Let $F \in C^2(\Gamma)$ be a concave curvature function, homogeneous of degree 1 with $F(1,\ldots, 1) > 0$, then 
	\begin{equation}
	F\leq \frac{F(1, \ldots, 1)}{n} H.
	\end{equation}
\end{Lemma}
\begin{proof}
	See \cite[Lemma 2.2.20]{GerhCP}.
\end{proof}
\begin{Bemerkung}
	To estimate tensors, we will need a Riemannian metric on $N$. We use a Riemannian reference metric, which we define by
	\begin{equation}
		\tilde{g}_{\alpha\beta}dx^\alpha dx^\beta = e^{2\psi} \left\{(dx^0)^2 + \si_{ij}dx^i dx^j\right\}.
	\end{equation}
	The corresponding norm of a vector field $\eta$ on $N$ will be denoted by
	\begin{equation}
		|||\eta||| = (\tilde{g}_{\alpha\beta}\eta^\alpha \eta^\beta)^{\frac{1}{2}}.
	\end{equation}
\end{Bemerkung}

Finally, we want to note that we will use the parabolic H\"older spaces in later sections, where for the notation we refer to \cite[Definition 2.5.2]{GerhCP}.

\section{Evolution equations}
In this chapter we state some facts about the representation of hypersurfaces as graphs and the evolution equations of the geometric quantities needed throughout the paper. For a derivation of the latter we refer to \cite[Chapter 2]{GerhCP}. Note that there is a slight but significant difference in the evolution equations compared to the Riemannian case due to the sign change in the Gau\ss{} equation.
%
%
%
%
%
First of all, we have in view of \cite[Lemma 3.1]{KroenerInv} the following
\begin{Lemma}
	Let $N$ be a smooth, connected, globally hyperbolic Lorentzian manifold with a compact, connected Cauchy hypersurface $\mathcal{S}_0$ and $M \subset N$ a compact, connected, spacelike hypersurface of class $C^{m, \alpha}$, $0\leq \alpha \leq 1$, $1\leq m \in \bbn$, then $M$ can be written as a graph over $\mathcal{S}_0$
	\begin{equation}
		M = \textnormal{graph }u_{|\mathcal{S}_0},
	\end{equation}
	with $u \in C^{m,\alpha}(\mathcal{S}_0)$.
\end{Lemma}
We remark that the additional regularity mentioned above follows by the same proof as in \cite[Lemma 3.1]{KroenerInv}, the implicit function theorem being the main theorem used in the proof. 

From now on we assume to work in local coordinates of the special coordinate system given by \eqref{GKS}.

The flow hypersurfaces can be written as graphs over $\mathcal{S}_0$
\begin{equation}
	M(t) = \{x^0 = u(x^i): x = (x^i) \in \mathcal{S}_0\},
\end{equation}
where we use the symbol $x$ ambiguously by denoting points $p = (x^\alpha) \in N$ as well as points $p = (x^i) \in \mathcal{S}_0$.
\noindent
Now suppose the flow hypersurfaces are given by an embedding $x = x(t, \xi)$, where $\xi = (\xi^i)$ are local coordinates of a compact manifold $M$, i.e. initially we have the embedding $x:M \rightarrow N$, $M_0 := x(M)$. Then there holds
\begin{equation}
\begin{split}
&x^0 = u(t, \xi) = u(t, x(t,\xi)),\\
&x^i = x^i(t, \xi).
\end{split}
\end{equation}
The induced metric has the form
\begin{equation}
\label{NotMetric}
g_{ij} = e^{2\psi}\{-u_i u_j + \si_{ij}\},
\end{equation}
where $\si_{ij}$ is evaluated at $(u(x), x)$. Its inverse $(g^{ij}) = (g_{ij})^{-1}$ can be expressed as
\begin{equation}
\label{InvMetric}
	g^{ij} = e^{-2\psi}\big\{\si^{ij} + \frac{\breve{u}^i}{v}\, \frac{\breve{u}^j}{v}\big\},
\end{equation}
where $(\si^{ij}) = (\si_{ij})^{-1}$ and we distinguish $u^i = g^{ij}u_j$ and
\begin{equation}
	\breve{u}^i = \si^{ij}u_j
\end{equation}
and where we define
\begin{equation}
\label{v}
	v^2 = 1 - \si^{ij}u_i u_j \equiv 1 - |Du|^2.
\end{equation}
Hence, graph $u$ is spacelike if and only if $|Du| < 1$, in view of \eqref{NotMetric}.

The past-directed normal has the form
\begin{equation}
\label{normal}
	(\nu^\alpha) = - v^{-1}e^{-\psi}(1, \breve{u}^i).
\end{equation}

Furthermore, looking at the component $\alpha = 0$ in the gaussian formula, we obtain
	\begin{equation}
	\label{EvHU}
		e^{-\psi}\tilde{v} \, h_{ij} = - u_{;ij} - \bar{\Gamma}^0_{00}u_i u_j - \bar{\Gamma}^0_{0i}u_j - \bar{\Gamma}^0_{0j}u_i - \bar{\Gamma}^0_{ij},
	\end{equation}
where the covariant derivatives are taken with respect to the induced metric of the considered hypersurface and $\tilde{v} = v^{-1}$. For later use, we reformulate the above expression as in \cite[(2.5.11)]{GerhCP}, such that
\begin{equation}
\label{EvHUlin}
	e^{-\psi}\tilde{v}\, h_{ij} = -v^{-2}u_{ij} + e^{-\psi}\bar{h}_{ij} - e^{-\psi} \psi_\alpha\bar{\nu}^{\alpha}\bar{g}_{ij}+ v^{-1}e^{-\psi}\psi_{\alpha}\nu^\alpha g_{ij},
\end{equation}
where $\bar{\nu}$, $\bar{g}_{ij}$ and $\bar{h}_{ij}$ denote the normal, metric and second fundamental form of the coordinate slices $\{x^0 = \textnormal{const}\}$ and where the covariant derivatives of $u$ are now taken with respect to the metric $\si_{ij}(u, x)$.

In the Lorentzian case controlling the $C^1$-norm of graph $u$ is tantamount to controlling $\tilde{v}$ in view of \eqref{v} and
\begin{equation}
	||Du||^2 = g^{ij}u_iu_j = e^{-2\psi}\frac{|Du|^2}{v^2}.
\end{equation}
Finally, as for the curvature flows with general curvature functions we had to assume the existence of a strictly convex function $\chi \in C^2(\bar{\Om})$ in a given domain $\Om$, we shall state a geometric condition guaranteeing the existence of such a function. For a proof of the following Lemma see \cite[Lemma 1.8.3]{GerhCP}.
\begin{Lemma}
\label{ExistenceStrictlyConvex}
	Let $N$ be a smooth, globally hyperbolic Lorentzian manifold, $\mathcal{S}_0$ a Cauchy surface, $(x^\alpha)$ a future directed Gaussian coordinate system associated with $\mathcal{S}_0$ and $\bar{\Om} \subset N$ compact. Then there exists a strictly convex function $\chi \in C^2(\bar{\Om})$, i.e. a function satisfying
	\begin{equation}
		\chi_{\alpha\beta} \geq c_0 \bar{g}_{\alpha\beta}
	\end{equation}
	with a positive constant $c_0$, provided the level hypersurfaces $\{x^0 = \textnormal{const}\}$ that intersect $\bar{\Om}$ are strictly convex.
\end{Lemma}
%
%
%
%
We consider a curvature function $F \in C^{m,\alpha}(\Gamma)$, $2 \leq m \in \bbn$, $0< \alpha <1$, a function $f = f(t)$ and a real function $\Phi \in C^{m,\alpha}(\bbr)$ and write from now on $\Phi = \Phi(F)$.

The curvature flow is then given by the evolution problem \eqref{floweq} with $f=f_k$, $0\leq k \leq n$, as defined in \eqref{globTerm} (where we remark again that not all values of $k$ are allowed for $F=H$ or $F=\si_2$, see the remarks after equation \eqref{globTerm}).

We will assume throughout the next sections that short time existence has already been assured and we consider a solution $x \in H^{m+\alpha, \frac{m+\alpha}{2}}(Q_{T^*})$ of the curvature flow on a maximal interval $[0, T^*)$, $0 < T^* \leq \infty$, where $Q_T = [0, T) \times M$. Short-time existence is well known for the curvature flow without the global term as we are dealing with a parabolic problem and with a fixed point argument we can extend this result to the flow including the global term. This will be supplemented in section \ref{shorttime}.

Hence we consider a sufficiently smooth solution of the initial value problem \eqref{floweq} and show how the geometric quantities of the hypersurfaces $M(t)$ evolve. All time derivatives are \textit{total} derivatives, i.e., covariant derivatives of tensor fields defined over the curve $x(t)$, cf. \cite[Chapter 11.5]{GerhAna}.
	
First, we consider the evolution equations for the hypersurfaces represented as graphs.

Looking at the component $\alpha = 0$ of the flow \eqref{floweq} we obtain the scalar flow equation
\begin{equation}
\label{totalU}
	\dot{u} = -e^{-\psi}v^{-1}(\Phi - f),
\end{equation}
where the time derivative is a total time derivative of $u = u(t, x(t, \xi))$. If however we consider $u$ to depend on $u = u(t, \xi)$ we obtain the partial derivative
\begin{equation}
\label{partialU}
\begin{split}
	\frac{\partial u}{\partial t} &= \dot{u} - u_k\dot{x}^k\\
	&= -e^{-\psi}v(\Phi - f).
\end{split}
\end{equation}
Let us now state the evolution equations, where we note that all covariant derivatives appearing in these equations are taken with respect to the induced metric of the flow hypersurfaces:
\begin{Lemma} We have the following evolution equations:
{\allowdisplaybreaks
\begin{align} 
\label{EvMetrik}
	&\dot{g}_{ij} = 2(\Phi - f) h_{ij}, \\
\label{EvVolumen}
	& \frac{d}{dt}\sqrt{g} = (\Phi - f) H \sqrt{g},\quad \textnormal{where } g = \det{g_{ij}},\\
\label{EvNormal}
	& \dot{\nu} = g^{ij}\Phi_i x_j,\\
\label{EvSecFF1}
	& \dot{h}^j_i = \Phi_{;i}^{\,j} - (\Phi - f)\{h^k_ih^j_k +\bar{R}_{\alpha\beta\gamma\delta}\nu^\alpha x_i^\beta\nu^\gamma x_k^\delta g^{kj}\},\\
\label{EvSecFF1b}
	& \dot{h}_{ij} = \Phi_{;ij} - (\Phi - f)\{- h^k_ih_{kj} + \bar{R}_{\alpha\beta\gamma\delta}\nu^\alpha x_i^\beta\nu^\gamma x_j^\delta\},\\
\label{EvPhi}
	& \dot{\Phi} - \Phi' F^{ij}\Phi_{;ij} = -\Phi' (\Phi - f)\{F^{ij}h_i^kh_{kj} + F^{ij}\bar{R}_{\alpha\beta\gamma\delta}\nu^\alpha x_i^\beta \nu^\gamma x_j^\delta\}, \\
	& \quad \textnormal{where } \Phi' = \frac{d}{dr} \Phi(r)\notag,\\
\label{EvParGraph}
	 &\dot{u} - \Phi' F^{ij}u_{ij} = -e^{-\psi}\tilde{v}(\Phi - f) + \Phi'Fe^{-\psi}\tilde{v}\\
	 &\quad + \Phi' F^{ij}\{\bar{\Gamma}^0_{00}u_iu_j + 2 \bar{\Gamma}^0_{0i}u_j + \bar{\Gamma}^0_{ij}\}\notag ,\\
\label{EvV}
	& \dot{\tilde{v}} - \Phi' F^{ij}\tilde{v}_{;ij} = - \Phi' F^{ij}h_{ik}h^k_j \tilde{v} + [(\Phi - f) - \Phi' F] \eta_{\alpha\beta} \nu^\alpha \nu^\beta\\
		& \quad - 2\Phi' F^{ij}h^k_j x_i^\alpha x_k^\beta \eta_{\alpha\beta} - \Phi' F^{ij}\eta_{\alpha \beta \gamma} x_i^\beta x_j^\gamma \nu^\alpha \notag\\
		& \quad - \Phi' F^{ij} \bar{R}_{\alpha\beta\gamma\delta} \nu^\alpha x_i^\beta x_k^\gamma x_j^\delta \eta_\epsilon x_l^\epsilon g^{kl}\notag ,\\
		& \quad \textnormal{where } \eta \textnormal{ is the covariant vector field } (\eta_\alpha) = e^\psi (-1, 0, ..., 0) \notag ,\\
\label{EvSecFF2}
	& \dot{h}^j_i - \Phi' F^{kl}h^j_{i;kl} = - \Phi' F^{kl}h_{rk}h^r_lh_i^j + \Phi' Fh_{ri}h^{rj}\\
		& \quad - (\Phi - f)h^k_ih_k^j + \Phi' F^{kl,rs}h_{kl;i}h_{rs;}^{\quad j} + \Phi'' F_iF^j \notag\\
		& \quad + 2\Phi' F^{kl}\bar{R}_{\alpha\beta\gamma\delta}x_m^\alpha x_i^\beta x_k^\gamma x_r^\delta h_l^mg^{rj} - \Phi' F^{kl}\bar{R}_{\alpha\beta\gamma\delta} x_m^\alpha x_k^\beta x_r^\gamma x_l^\delta h_i^mg^{rj} \notag\\
		& \quad -\Phi' F^{kl}\bar{R}_{\alpha\beta\gamma\delta}x_m^\alpha x_k^\beta x_i^\gamma x_l^\delta h^{mj} - \Phi' F^{kl}\bar{R}_{\alpha\beta\gamma\delta}\nu^\alpha x_k^\beta \nu^\gamma x_l^\delta h_i^j \notag\\ 
		& \quad + \Phi' F \bar{R}_{\alpha\beta\gamma\delta}\nu^\alpha x_i^\beta \nu^\gamma x_m^\delta g^{mj} - (\Phi - f) \bar{R}_{\alpha\beta\gamma\delta}\nu^\alpha x_i^\beta \nu^\gamma x_m^\delta g^{mj} \notag\\
		& \quad + \Phi' F^{kl}\bar{R}_{\alpha\beta\gamma\delta ;\epsilon}\{\nu^\alpha x_k^\beta x_l^\gamma x_i^\delta x_m^\epsilon g^{mj} + \nu^\alpha x_i^\beta x_k^\gamma x_m^\delta x_l^\epsilon g^{mj}\}.\notag
\end{align}
}
\end{Lemma}
\begin{proof}
	See \cite[Lemma 2.3.1 for the first two equations, then Lemma 2.3.2, Lemma 2.3.3 for the next two equations, Lemma 2.3.4, equation \eqref{partialU} together with \eqref{EvHU}, Lemma 2.4.4, Lemma 2.4.1]{GerhCP}.
\end{proof}

\section{Height estimates and volume preservation of the flow}
\label{C0}
First, we remind the definition of the mixed volume, for $k \in \{0, \ldots, n\}$ and a hypersurface $M$ represented by a graph $u$ we have:
%
%
%
%
\begin{equation}
	V_{n+1-k} = 
	\begin{cases}
	\int_{\mathcal{S}_0}{\int_0^u{{e^{\psi}\sqrt{\bar{g}}}}} , & k = 0\\
	\{(n+1) {n \choose k} \}^{-1}\int_{M}{H_{k-1}\, \mathrm{d\mu}}, & k = 1, \ldots, n,
	\end{cases}
\end{equation}
where for $(t, x) \in N$ we denote by $\bar{g}(t, x) = \det(\bar{g}_{ij})(t,x)$ the volume element of the level hypersurface $x^0 = t$ at the point $x \in \calS_0$. The choice of the reference point $0 \in (a,b)$ for the enclosed volume is arbitrary. Now we are going to prove the claimed volume preservation property of the flow \eqref{floweq} with $f=f_k$, $0\leq k\leq n$.
%
%
%
%
\begin{Lemma}
\label{VolPreservation}
	For $k=0$ the enclosed volume $V_{n+1}$ and for $k=1$ the volume of the hypersurfaces $V_n$ is preserved.\\
	If the ambient space has vanishing sectional curvatures, then for $1 <k \leq n$ the mixed volume $V_{n+1-k}$ is preserved.
\end{Lemma}
\begin{proof}
	First we observe that for $x\in \mathcal{S}_0$ we have 
	\begin{equation}
		\sqrt{g(u(x), x)} = v\sqrt{\det(\bar{g}_{ij}(u(x), x))},
	\end{equation}
	where $\bar{g}_{ij}(t, \cdot)$ denotes the metric of the level hypersurface $x^0 = t$.
	
	Taking this into account, we have for $k=0$ in view of \eqref{partialU}:
	\begin{equation}
	\begin{split}
		\frac{d}{dt} V_{n+1} &= \int_{\mathcal{S}_0}{\frac{\partial u}{\partial t}\,e^\psi\sqrt{\bar{g}(u(x), x)} \, \mathrm{dx}}\\
		&= -\int_{\mathcal{S}_0}{(\Phi - f_0) \sqrt{g(u(x),x)}\,\mathrm{dx}} = 0,
	\end{split}
	\end{equation}
	in view of the definition of $f_0$.
	Hence the enclosed volume is preserved by the flow.
	
	For $k=1$ we have in view of \eqref{EvVolumen}
	\begin{equation}
	\begin{split}
		(n+1)n \frac{d}{dt} V_n = \frac{d}{dt} |M_t|
		= \int_{M_t}{(\Phi - f_1) H\, \mathrm{d\mu_t}} = 0.
	\end{split}
	\end{equation}
	Finally, for $1 < k \leq n$ we assume the ambient space has vanishing sectional curvatures. Then we exploit Lemma \ref{divFree} and Lemma \ref{symPol}.\\
	We get
	\begin{equation}
	\begin{split}
	(n+1){n \choose k}&\frac{d}{dt} \int_{M_t}{H_{k-1} \mathrm{d\mu_t}} = \int_{M_t}{kH_k(\Phi - f_k) \mathrm{d\mu_t}} \\
	&- \int_{M_t}{(\Phi - f_k)(H_{k-1})^i_j\bar{R}_{\alpha\beta\gamma\delta}\nu^\alpha x_i^\beta \nu^\gamma x_k^\delta g^{kj} \mathrm{d\mu_t}}\\
	&= k \int_{M_t}{(\Phi - f_k) H_k \,\mathrm{d\mu_t}} = 0.
	\end{split}
	\end{equation}
\end{proof}
%
%
%
%
In order to prove the $C^0$-estimates we will show that the curvature function is bounded during the evolution. Together with the monotonicity of the constant $F$-curvature hypersurfaces, which we will prove afterwards, we then obtain that the flow stays within the domain bounded by the barriers for all times.
\begin{Lemma}
\label{FBoundsLemma}
Let
\begin{equation}
c_1 := \underset{M_0}{\min} \,\Phi(F) \quad \textnormal{ and }\quad c_2 := \underset{M_0}{\max} \,\Phi(F),
\end{equation}
then there holds for all times $0 < t < T^*$
\begin{equation}
\label{FBounds}
c_1 \leq \Phi(F) \leq c_2.
\end{equation}
Moreover $\Phi^{\sup}(t) := \underset{M_t}{\max} \,\Phi$ is monotonically decreasing and $\Phi^{\inf}(t) := \underset{M_t}{\min} \,\Phi$ is monotonically increasing.
\end{Lemma}
\begin{proof}
We will prove \eqref{FBounds} holds until $T_0$, where $0 < T_0 < T^*$ is arbitrary. This will prove the first statement and the second one follows by observing that the argument holds as well for the interval $[t_1, T^*)$, where $0 < t_1 < T^*$ is arbitrary.

We only prove the upper bound, the proof for the lower bound follows analogously.

Let
\begin{equation}
\tilde{\Phi} := (\Phi(F) - c_2) - \eps t - \eps,
\end{equation}
where $\eps >0$ is chosen arbitrarily. Therefore there holds
\begin{equation}
\tilde{\Phi}_{|t=0} < 0
\end{equation}
and from \eqref{EvPhi} we get
\begin{equation}
\label{EvPhiTilde}
	\dot{\tilde{\Phi}} - \Phi' F^{ij} \tilde{\Phi}_{;ij} = -\Phi' (\Phi - f) \{F^{ij}h_{ik}h^k_j + F^{ij}\bar{R}_{\alpha\beta\gamma\delta}\nu^\alpha x_i^\beta \nu^\gamma x_j^\delta\} - \eps.
\end{equation}
Suppose there is a point $(t_0, x_0)$, $0 < t_0 \leq T_0$, $x_0 \in \mathcal{S}_0$ such that
\begin{equation}
	\tilde{\Phi}(t_0, x_0) = 0
\end{equation}
and where $t_0$ is the first time for this to happen.

Hence there holds for all $x\in \mathcal{S}_0$
\begin{equation}
	\tilde{\Phi}(t_0, x) \leq \tilde{\Phi}(t_0, x_0)
\end{equation}
and this implies by the definition of $\tilde{\Phi}$ and $f_k$
\begin{equation}
f_{k}(t_0) \leq \Phi(t_0, x_0).
\end{equation}
Evaluating \eqref{EvPhiTilde} at $(t_0, x_0)$ yields therefore in view of the maximum principle
\begin{equation}
	0 \leq - \eps < 0,
\end{equation}
where we used the timelike convergence condition \eqref{TCC} in case of the mean curvature flow and \eqref{NPTSC} for general curvature functions as well as the non-negativity of the term $F^{ij}h_{ik}h^{k}_j$.

This contradiction implies
\begin{equation}
	\Phi < \eps t + \eps + c_2 \leq \eps T_0 + \eps + c_2
\end{equation}
and the Lemma follows because $\eps$ can be chosen arbitrarily small.
\end{proof}
%
%
%
%
Note that this Lemma ensures in the case of $F=H$ or $F=\si_2$, that the principal curvatures of the flow lie in the cone of definition, where we use Lemma \ref{FHineq} for $F= \si_2$. Furthermore, with regard to the supplementary function, the proof merely depends on the fact that $\Phi' > 0$.

We also want to point out the following observation, which can be used to prove long time existence for bounded, but otherwise more general global force terms, than the ones we consider in this paper, as long as we have suitable barriers:
\begin{Bemerkung}
\label{FBoundsRem}
	The proof of the preceding Lemma shows that for a global term $f = f(t)$ which is bounded from below and above by $c_1$ and $c_2$ respectively, where these are arbitrary constants, the curvature function $\Phi(F)$ is bounded by the same constants if $\Phi(F)_{|t=0}$ is.
\end{Bemerkung}
%
%
%
%
Now we want to show the monotonicity of hypersurfaces with respect to their $F$-curvature. For the mean curvature in a cosmological spacetime the proof can be found in \cite[Lemma 4.7.1]{GerhCP}, for general $F$ the proof needs a minor modification:
\begin{Lemma}
\label{FMonotone}
Let $N$ be a smooth cosmological spacetime with compact Cauchy-hypersurface $\calS_0$ and non-positive timelike sectional curvatures, $F$ a strictly monotone curvature function defined on an open, convex, symmetric cone $\Gamma$, $F \in C^1(\Gamma)\cap C^0(\bar{\Gamma})$, such that $F$ vanishes on the boundary of $\Gamma$ and $F > 0$ in $\Gamma$. 

Let $M_i = \text{graph}\, u_i$, $i=1,2$ be two compact, connected, spacelike, admissible hypersurfaces of class $C^2$, such that the respective $F$-curvatures $F_i$ satisfy
\begin{equation}
\label{F1lessF2}
	F_1 < (\le) \,\underset{M_2}{\min} \,F_2,
\end{equation}
and if $p \in \calS_0$ is a point such that $F_1(u_1(p), p) = \underset{M_2}{\min} \, F_2$, then we assume that the principal curvatures of $M_2$ at $(u_2(p), p)$ are not all equally to zero.

Then there holds
\begin{equation}
	u_1 < (\leq)\, u_2.
\end{equation}
\end{Lemma}
\begin{proof}
	In view of the equation \eqref{EvHU} and the maximum principle it suffices to show
	\begin{equation}
	\label{u1kleineru2}
		u_1 \leq u_2.
	\end{equation}
	Now suppose \eqref{u1kleineru2} is not valid, so that
	\begin{equation}
		E(u_1) = \{x \in \mathcal{S}_0: u_2(x) < u_1(x)\} \neq \emptyset.
	\end{equation}
	Thus there exist points $p_i \in M_i$ such that
	\begin{equation}
		0 < d_0 = d(M_2, M_1) = d(p_2, p_1) = \sup\,\{d(p,q): (p,q) \in M_2 \times M_1\},
	\end{equation}
	where $d$ is the Lorentzian distance function, which is finite and continuous in our setting, see \cite[Theorem 12.5.9]{GerhAna}.
	
	Now let $\phi$ be a maximal geodesic from $M_2$ to $M_1$ realizing this distance with endpoints $p_2$ and $p_1$ and parametrized by arc length. 
	
	Denote by $\bar{d}$ the Lorentzian distance function to $M_2$, i.e. for $p\in I^+(M_2)$
	\begin{equation}
		\bar{d}(p) = \underset{q\in M_2}{\sup} d(q, p).
	\end{equation}
	Since $\phi$ is maximal, $\Lambda = \{\phi(t): 0\leq t < d_0\}$ contains no focal points of $M_2$, cf \cite[Theorem 34, p.285]{ONeill}, hence there exists an open neighbourhood $\Pi = \Pi(\Lambda)$ such that $\bar{d}$ is of class $C^2$ in $\Pi$, cf \cite[Theorem 1.9.15]{GerhCP} and $\Pi$ is part of the largest tubular neighbourhood of $M_2$ and hence covered by an associated normal Gaussian coordinate system $(x^\alpha)$ satisfying $x^0 = \bar{d}$ in $\{x^0 >0\}$, see \cite[Theorem 1.9.22]{GerhCP}.
	
	In this coordinate system $M_2$ is the level set $\{\bar{d} = 0\}$ and the level sets
	\begin{equation}
		M(t) = \{p \in \Pi: \bar{d}(p) = t\}
	\end{equation}
	are $C^2$-hypersurfaces.
	
	Next we want to derive a formula for the evolution of the $F$-curvature of the level hypersurfaces of this coordinate system. Let us define the flow of the level hypersurfaces by
	\begin{align}
		&\dot{x} = - \nu,\\
		&x(0) = x_0,
	\end{align}
	where $x_0$ is the embedding of $M_2$. Then we infer from \cite[Proposition 1.9.4]{GerhCP}
	\begin{equation}
	\label{Tubenumg}
		x, \,\dot{x} \in C^1((-\eps_0, d_0) \times B_\rho(\xi_0)),
	\end{equation}
	where $\eps_0 >0$, $x_0(\xi_0) = \phi(0)$ and $(-\eps_0, d_0) \times x_0(B_\rho(\xi_0)) \subset \Pi$. 
	Hence $g_{ij} = \langle x_i, x_j \rangle$ as well as $h_{ij} = - \frac{1}{2} \dot{g}_{ij}$ (see equation \eqref{EvMetrik} with $f=1$) are continuously differentiable with respect to space and time. From \eqref{EvSecFF1} we then obtain the equation
	\begin{equation}
		\dot{h}_j^i = h^i_k h^k_j + \bar{R}_{\alpha \beta \gamma \delta}\nu^\alpha x_k^\beta \nu^\gamma x_j^\delta g^{ki},
	\end{equation}
	where we note that in view of \eqref{Tubenumg} one can verify that $\bar{R}_{\alpha \beta \gamma \delta}\nu^\alpha \nu^\gamma = \bar{R}_{0\beta 0\delta}$ is continuous in this coordinate system, since $\bar{\Gamma}^0_{\alpha \beta} = \frac{1}{2}\bar{g}_{\alpha \beta,0}$.
	
	For the $F$-curvature of $M(t)$ we obtain then the equation 
	\begin{equation}
	\label{FRise}
		\dot{\bar{F}} = \bar{F}^l_k\bar{h}_l^m \bar{h}_m^k + \bar{F}_k^l\bar{R}_{\alpha\beta\gamma\delta}\nu^\alpha x_l^\beta \nu^\gamma x_m^\delta \bar{g}^{mk},
	\end{equation}
	where the geometric quantities like $\bar{g}_{ij}$, $\bar{h}_{ij}$ and so on denote the geometric quantities of the level hypersurfaces and they are not to be confused with the quantities of the ambient space.
	This implies that the $F$-curvature of $M(t)$ is monotonically increasing with respect to $t$ in view of the strict monotonicity of the $F$-curvature, hence the level hypersurfaces are admissible, since $F$ vanishes only on $\partial \Gamma$.

	Next, consider a tubular neighbourhood $\mathcal{U}$ of $M_1$ with corresponding normal Gaussian coordinate system $(x^\alpha)$. The level sets
	\begin{equation}
	\tilde{M}(s) = \{ x^0 = s\}\quad, -\delta < s < 0,
	\end{equation}
	lie in the past of $M_1 = \tilde{M}(0)$ and are all of class $C^2$ for small $\delta$.
	
	Since the geodesic $\phi$ is normal to $M_1$, it is also normal to $\tilde{M}(s)$ and the length of the geodesic segment of $\phi$ from $\tilde{M}(s)$ to $M_1$ is exactly $-s$, thus equal to the distance from $\tilde{M}(s)$ to $M_1$, hence we deduce
	\begin{equation}
		d(M_2, \tilde{M}(s)) = d_0 + s.
	\end{equation}
	We infer that $\{\phi(t): 0\leq t \leq d_0 +s \}$ also represents a maximal geodesic from $M_2$ to $\tilde{M}(s)$ and we conclude further that, for fixed $s$, the hypersurface $\tilde{M}(s) \cap \Pi$ is contained in the past of $M(d_0+s)$ and touches $M(d_0+s)$ in $p_s = \phi(d_0 +s)$. 
	
	Hence by the maximum principle there holds
	\begin{equation}
		F_{|\tilde{M}(s)}(p_s) \geq F_{|M(d_0+s)}(p_s). 
	\end{equation}
	Furthermore, if 
	\begin{equation}
	\label{FGleichheit}
		F_{1|\phi(0)} = \underset{M_2}{\min}\, F_2,
	\end{equation}
	then by using the additional assumption we conclude that if we choose $\delta > 0$ small enough, then in view of \eqref{FRise} there exists $\eps > 0$ not depending on $s$, $-\delta < s < 0$, such that there holds
	\begin{equation}
		F_{|M(d_0+s)}(p_s) > \underset{M_2}{\min}\,F_2 + \eps.
	\end{equation}
	On the other hand the $F$-curvature of $\tilde{M}(s)$ converges to the $F$-curvature of $M_1$ if $s$ tends to zero, hence we conclude
	\begin{equation}
	\label{Finequality}
		F_1(p_1) \geq \underset{M_2}{\min}\, F_2 + \eps,
	\end{equation}
	where $\eps > 0$ if \eqref{FGleichheit} is satisfied (otherwise it can be equal to zero), yielding in either case a contradiction to \eqref{F1lessF2}.
	
\end{proof}

%
%
%
%
The barrier condition and the preceding Lemmata imply the following
\begin{Proposition}
	If $u_i = $ graph $M_i$, $i = 1,2$, where $M_1$ and $M_2$ denote the lower and upper barrier respectively, then there holds
	\begin{equation}
		u_1 \leq u(t) \leq u_2.
	\end{equation}
\end{Proposition}
In the case $F= H$, $\Gamma = \bbr^n$, this Proposition follows by the proof of Lemma \ref{FMonotone}, since now all appearing hypersurfaces are admissible.

%
%
%
%
Lemma \ref{FMonotone} also yields the uniqueness of constant $F$-curvature hypersurfaces:
\begin{Korollar}
	Let $N$ be as in Lemma \ref{FMonotone} and $F$ be a curvature function of class $(K)$ defined on $\Gamma_+$ or $F= \si_k$, $1\leq k \leq n$, defined on $\Gamma_k$. Then a compact, connected, spacelike hypersurface of class $C^2$ with $F \equiv c$ for some constant $c > 0$, is uniquely determined.
\end{Korollar}

\section{Gradient estimates}
Let $\Phi$ be a function in $C^{2,\alpha}(\bbr)$, which satisfies
\begin{equation}
	\Phi' > 0 \quad \text{ and }\quad \Phi'' \leq 0.
\end{equation}
Let $f =f(t)$ be a bounded function and suppose $C^0$-estimates have already been established, i.e. the flow stays inside a compact region $\bar{\Om} \subset N$. Let $F$ be a curvature function which is monotone, concave and homogeneous of degree 1. Then we show,
following the proof in \cite[Section 5]{EnzDiplom},
\begin{Proposition}
\label{C1estimates}
	During the evolution of the flow \eqref{floweq} the term $\tilde{v}$ is uniformly bounded:
	\begin{equation}
		\tilde{v} \leq c = c(\Om, |\Phi|_0, |\Phi'|_0).
	\end{equation} 
\end{Proposition}
We can allow for such a general supplementary function $\Phi$ as above, because we already established bounds for $\Phi$ in Lemma \ref{FBoundsLemma}.

First we need some Lemmata:
\begin{Lemma}
	The composite function 
	\begin{equation}
	\label{C1phi}
		\phi = e^{\mu e^{\lambda u}},
	\end{equation}
	where $\mu$, $\lambda$ are constants, satisfies the equation
	\begin{equation}
	\begin{split}
		\dot{\phi} - \Phi' F^{ij}\phi_{ij} = &\,e^{-\psi}\tilde{v}\{\Phi'F - \Phi + f\} \mu\lambda e^{\lambda u} \phi \\
		&+\Phi'F^{ij}\{\bar{\Gamma}^0_{00}u_iu_j + 2 \bar{\Gamma}^0_{0i}u_j + \bar{\Gamma}^0_{ij}\}\mu\lambda e^{\lambda u}\phi \\
		&- [1+ \mu e^{\lambda u}] \Phi'F^{ij}u_i u_j\mu \lambda^2 e^{\lambda u}\phi.
	\end{split}
	\end{equation}
\end{Lemma}
The Lemma follows from \eqref{EvParGraph}.\\
For a proof of the following two Lemmata we refer to \cite{GerhScalar}:
\begin{Lemma}
\label{C1PrepLemma}
	There is a constant $c=c(\Om)$ such that for any positive function $0 < \eps = \eps(x)$ on $\mathcal{S}_0$ and any hypersurface $M(t)$ of the flow we have
\begin{align}
	|||\nu||| &\leq c \tilde{v},\\
	g^{ij} &\leq c \tilde{v}^2 \si^{ij},\\
	F^{ij} &\leq F^{kl}g_{kl}g^{ij},\\
	|F^{ij}h_j^kx_i^\alpha x_k^\beta \eta_{\alpha\beta}| &\leq \frac{\eps}{2}F^{ij}h_i^kh_{kj}\tilde{v} + \frac{c}{2\eps}F^{ij}g_{ij}\tilde{v}^3,\\
	|F^{ij}\eta_{\alpha\beta\gamma}x_i^\beta x_j^\gamma \nu^\alpha| &\leq c \tilde{v}^3 F^{ij}g_{ij},\\
	|F^{ij}\bar{R}_{\alpha\beta\gamma\delta}\nu^\alpha x_i^\beta x_k^\gamma x_j^\delta \eta_\eps x_l^\eps g^{kl}| &\leq c\tilde{v}^3 F^{ij}g_{ij},\\
	\label{C1VectorEst}
	|||x_i^\alpha\xi^i||| &\leq c \tilde{v} \quad \forall (p, \xi) \in TM(t).
\end{align}
\end{Lemma}
\begin{Lemma}
\label{C1EstDerivatives}
	Let $M \subset \bar{\Om}$ be a graph over $\mathcal{S}_0$, $M = $ graph $u$, and $\eps = \eps(x)$ a function defined in $\mathcal{S}_0$, $0< \eps < \frac{1}{2}$. Let $\phi$ be defined through
	\begin{equation}
		\phi = e^{\mu e^{\lambda u}},
	\end{equation}
	where $0< \mu$ and $\lambda < 0$. Then there exists $c = c(\Om, |\Phi|_0, |\Phi'|_0)$ such that
	\begin{equation}
	\begin{split}
		2\Phi'|F^{ij}\tilde{v}_i\phi_j| \leq &c F^{ij}g_{ij} \tilde{v}^3 |\lambda|\mu e^{\lambda u} \phi + (1 - 2\eps) \Phi' F^{ij} h_i^kh_{kj} \tilde{v}\phi \\
		& + \frac{1}{1 - 2\eps} \Phi' F^{ij}u_i u_j \mu^2\lambda^2 e^{2\lambda u} \tilde{v}\phi.
	\end{split} 
	\end{equation}
\end{Lemma}
Now we can prove Proposition \ref{C1estimates}:
\begin{proof}
	We consider the function 
	\begin{equation}
		w = \tilde{v} \phi,
	\end{equation}
	where $\phi$ is chosen as in \eqref{C1phi} and we will choose
	\begin{equation}
		\mu = \frac{1}{4},	
	\end{equation}
	$\lambda$ negative with $|\lambda|$ large enough. Showing that $w$ is bounded is tantamount to show that $\tilde{v}$ is bounded.
	
	Let us furthermore assume that $u \leq -1$, for otherwise we could replace $u$ in the definition of $\phi$ by $(u - c)$, where $c > 1 + |u|$. We derive in view of the evolution equations and Lemma \ref{FBoundsLemma}, \ref{C1PrepLemma} and \ref{C1EstDerivatives} the parabolic inequality
	\begin{equation}
	\label{C1EvW}
	\begin{split}
		\dot{w} - \Phi'F^{ij}w_{ij} \leq & -\eps \Phi' F^{ij}h_i^kh_{kj}\tilde{v}\phi + c[\eps^{-1} + |\lambda|\mu e^{\lambda u}] F^{ij}g_{ij} \tilde{v}^3 \phi \\
		&+ [\frac{1}{1-2\eps} - 1] \Phi'F^{ij}u_iu_j \mu ^2 \lambda ^2 e^{2\lambda u}\tilde{v} \phi \\
		&- \Phi' F^{ij}u_i u_j \mu \lambda^2 e^{\lambda u}\tilde{v}\phi\\
		&+ c [|\eta_{\alpha\beta} \nu^\alpha\nu^\beta| + e^{-\psi}\mu\lambda e^{\lambda u}\tilde{v}^2]\phi,
	\end{split}
	\end{equation}
	where the function $0 < \eps = \eps(x) < \frac{1}{2}$ is the one chosen in Lemma \ref{C1PrepLemma}.
	
	We use the maximum principle to show that $w$ is bounded, let $0 < T < T^*$ and $x_0 = x(t_0, \xi_0)$ be such that
	\begin{equation}
		\underset{[0,T]}{\sup} \underset{M(t)}{\sup} w = w(t_0, \xi_0).
	\end{equation}
	We choose a coordinate system $(\xi^i)$ such that in the critical point 
	\begin{equation}
		g_{ij} = \delta_{ij}\qquad \text{and} \qquad h_i^k = \kappa_i \delta_i^k,
	\end{equation}
	and there holds
	\begin{equation}
	\label{C1PCup}
		\kappa_1 \leq \kappa_2 \leq \ldots \leq \kappa_n.
	\end{equation}
	Now assume $\tilde{v}(x_0) \geq 2$ and let $i = i_0$ be an index such that
	\begin{equation}
	\label{C1GradientU}
		|u_{i_0}|^2 \geq \frac{1}{n} ||Du||^2.
	\end{equation}
	We set $(e^i) = \frac{\partial}{\partial \xi^{i_0}}$ and assume without loss of generality that $0 < u_ie^i$. At $x_0$ there holds $Dw(x_0) = 0$, hence taking the scalar product with $(e^i)$ yields
	\begin{equation}
	\begin{split}
		-\tilde{v}_ie^i &= \mu\lambda e^{\lambda u}\tilde{v}u_ie^i\\
			& = e^\psi h_i^k u_k e^i - \eta_{\alpha\beta}\nu^\alpha x_i^\beta e^i,
	\end{split}
	\end{equation}
	where the second equation follows from $\tilde{v} = \eta_\alpha \nu^\alpha$ and the Weingarten equation (we remind that $\eta = e^\psi(-1, 0, \ldots, 0)$).
	
	Rearranging the terms and taking \eqref{C1VectorEst} as well as \eqref{C1GradientU} into account, we get for large $|\lambda|$
	\begin{equation}
		\kappa_{i_0} \leq [\mu\lambda e^{\lambda u} + c]\tilde{v} e^{-\psi} \leq \frac{1}{2}\mu\lambda e^{\lambda u}\tilde{v}e^{-\psi}.	
	\end{equation}
	Hence it follows that $\kappa_{i_0}$ is negative and of the same order as $\tilde{v}$, which already finishes the proof, if we have a lower bound for the principal curvatures.
	
	Next, considering the special coordinate system chosen above and the fact that $\kappa_{i_0}$ is negative, we conclude
	\begin{equation}
		-F^{ij}h_i^kh_{kj} \leq -\sum_{i=1}^{i_0} F_i\kappa_i^2 \leq - \sum_{i=1}^{i_0}F_i \kappa_{i_0}^2.
	\end{equation}
	Since $F$ is concave, we have at $x_0$
	\begin{equation}
	\label{C1FDeriv}
		F_1 \geq F_2\geq \ldots \geq F_n,
	\end{equation}
	hence there holds
	\begin{equation}
		-\sum_{i=1}^{i_0} F_i \leq - F_1 \leq -\frac{1}{n} \sum_{i=1}^n F_i.
	\end{equation}
	We conclude
	\begin{equation}
		-F^{ij}h_i^kh_{kj} \leq - c F^{ij}g_{ij}\mu^2 \lambda^2e^{2\lambda u}\tilde{v}^2.
	\end{equation}
	Inserting this estimate in \eqref{C1EvW} yields at $x_0$ with the choice $\eps = e^{-\lambda u}$:
	\begin{equation}
	\begin{split}
	0 &\leq - c F^{ij}g_{ij}\mu^2 \lambda^2e^{\lambda u}\tilde{v}^3\phi + cF^{ij}g_{ij}\mu |\lambda|e^{\lambda u}\tilde{v}^3\phi\\
	& + \frac{2}{1-2\eps}\Phi' F^{ij}u_iu_j \mu^2\lambda^2 e^{\lambda u}\tilde{v}\phi - \Phi' F^{ij}u_iu_j\mu\lambda^2 e^{\lambda u}\tilde{v}\phi\\
	& + c \mu |\lambda|e^{\lambda u} \tilde{v}^2 \phi.
	\end{split}
	\end{equation}
	The second row is negative due to the choice of $\mu$. The first term is the dominant one, if we choose $|\lambda|$ large enough, hence the right hand side is negative, which implies that the maximum of $w$ cannot occur at a point where $\tilde{v} \geq 2$.
\end{proof}

\section{Curvature estimates}
In this section we prove the boundedness of the principal curvatures during the flow, which together with the estimates in the next section will imply the long time existence of the flow by well-known arguments.

Now in view of Lemma \ref{FBoundsLemma} we are in an expedient situation, as $C^2$-estimates can be derived in the same way as for a constant force term $f \equiv c$. Nevertheless we will provide them for the sake of completeness.

First, we provide the curvature estimates for $F=H$, cf. \cite[Lemma 4.4.1]{GerhCP}:
\begin{Proposition}
	The principal curvatures of the flow \eqref{floweq} with curvature function $F = H$ and supplementary function $\Phi(x) = x$ for $k=0$ or $\Phi \in C^{m,\alpha}(\bbr_+, \bbr)$ an arbitrary function satisfying $\Phi'>0$ and $\Phi'' \leq 0$ for $k = 1$, are uniformly bounded during the flow.
\end{Proposition}
\begin{proof}
	Let $\zeta$ be defined by
	\begin{equation}
		\zeta = \sup\,\{ h_{ij}\eta^i\eta^j: ||\eta|| = 1\}.
	\end{equation}
	Let $0 < T < T^*$ and $x_0 = x_0(t_0)$ with $0 < t_0 \leq T$ be a point in $M(t_0)$ such that
	\begin{equation}
	\label{C2SupZeta}
		\underset{M_0}{\sup } \,\zeta < \sup\, \{ \underset{M_t}{\sup } \,\zeta:0 <t \leq T\} = \zeta(x_0).
	\end{equation}
	At first, we follow the usual argument, which allows one to substitute $\zeta$ by $h_n^n$ and use the evolution equation for the latter quantity to estimate $\zeta$:
	
	We choose Riemannian normal coordinates $(\xi^i)$ at $x_0 \in M(t_0)$ such that at this point we have
	\begin{equation}
		g_{ij} = \delta_{ij} \qquad \text{and}\qquad \zeta = h^n_n = \kappa_n,
	\end{equation}
	where we assume the principal curvatures are labelled as in \eqref{C1PCup}.
	
	Let $\tilde{\eta} = (\tilde{\eta}^i)$ be the contravariant vector field defined by 
	\begin{equation}
	\tilde{\eta} = (0, \ldots, 0, 1),
	\end{equation}
	and set 
	\begin{equation}
		\tilde{\zeta} = \frac{h_{ij}\tilde{\eta}^i\tilde{\eta}^j}{g_{ij}\tilde{\eta}^i\tilde{\eta}^j}.
	\end{equation}
	We note that $\tilde{\zeta}$ is well defined in a neigbourhood of $(t_0, x_0)$ and $\tilde{\zeta}$ assumes its maximum at $(t_0, x_0)$ as well. Moreover at $(t_0, x_0)$ we have 
	\begin{equation}
		\dot{\tilde{\zeta}} = \dot{h}^n_n
	\end{equation}
	and the spatial derivatives do also coincide. Hence at $(t_0, x_0)$ the function $\tilde{\zeta}$ satisfies the same differential equation as $h^n_n$. For the sake of greater clarity, we will treat therefore $h_n^n$ like a scalar and pretend that $\zeta$ is defined by
	\begin{equation}
		\zeta = \log h^n_n.
	\end{equation}
	In view of the maximum principle and Lemma \ref{FBoundsLemma} we deduce that there holds at $(t_0, x_0)$
	\begin{equation}
		0 \leq - \Phi' ||A||^2h^n_n + c H|h^n_n|^2 + c(1+ h^n_n),
	\end{equation}
	where $||A||^2 = h_{ij}h^{ij} = \sum_{i=1}^n \kappa_i^2$. This proves that $\zeta$ is bounded and since we already have a lower bound on $H$ we are done.
\end{proof}
Before we prove the next estimates, let us state the following 
\begin{Bemerkung}
	Let $\chi \equiv \chi_{\Om}$ be the strictly convex function, where we assume $\bar{\Om}$ is the region determined by the $C^0$-estimates. Then there exist constants $c= c(|\Phi|, |\Phi'|)$ and $c_0 > 0$ (depending on $|\Phi'|$ and the strict convexity of $\chi$), such that
	\begin{equation}
		\dot{\chi} - \Phi'F^{ij}\chi_{ij} \leq c\chi_\alpha\nu^\alpha - c_0F^{ij}g_{ij}.
	\end{equation}
\end{Bemerkung}
Next, we treat the case $F= \si_2$. The proof is as in \cite{EnzDiplom} :
\begin{Proposition}
	The principal curvatures of the flow \eqref{floweq} with $F = \si_2$, $\Phi(x) = x$ or $\Phi(x) = -x^{-1}$, are uniformly bounded during the flow, provided there exists a strictly convex function $\chi \in C^2(\bar{\Om})$.
\end{Proposition}
\begin{proof}
	Let $\zeta$ and $w$ be respectively defined by
	\begin{equation}
	\begin{split}
		&\zeta = \sup \{h_{ij}\eta^i\eta^j: ||\eta|| = 1\},\\
		\label{C2EvW}
		& w = \log \zeta + \lambda \chi,
	\end{split}
	\end{equation}
	where $\lambda > 0$ is a large constant. We will show that $w$ is bounded, if we choose $\lambda$ sufficiently large.
	
	Let $0 < T < T^*$ and $x_0 = x_0(t_0)$ with $0 < t_0 \leq T$ be a point in $M(t_0)$ such that
	\begin{equation}
	\label{C2SupW}
		\underset{M_0}{\sup }\, w < \sup \, \{ \underset{M_t}{\sup }\, w:0 <t \leq T\} = w(x_0).
	\end{equation}
	By the same procedure as in the last proof we introduce normal coordinates at $x_0 = x(t_0, \xi_0)$, and we may define $w$ by
	\begin{equation}
		w = \log h_n^n + \lambda \chi.
	\end{equation}
	If we assume $h^n_n$ and $\lambda$ to be greater than $1$, we deduce the following inequality at $(t_0, \xi_0)$
	\begin{equation}
	\label{C2EvWEst}
	\begin{split}
		0 &\leq -c F^{ij}h_{ik}h^k_j + c h^n_n + cF^{ij}g_{ij} + \lambda c -\lambda c_0 F^{ij}g_{ij} \\
		&+ \Phi' F^{ij}(\log h_n^n)_i (\log h_n^n)_j + \Phi' \frac{2}{\kappa_n - \kappa_1}\sum_{i=1}^n(F_n - F_i)(h_{ni;}^{\quad n})^2(h_n^n)^{-1},
	\end{split}
	\end{equation}
	where we have estimated bounded terms by a constant $c$ and the last term is due to the term with the second derivatives of $F$ in the evolution equation of $h_n^n$ (we use formula (2.1.72) in \cite{GerhCP} and note that the two parts in this formula are both negative for a concave curvature function, see \cite[Proposition 2.1.23]{GerhCP}).
	
	We distinguish two cases:
	
		\begin{em}Case 1.\end{em}\quad Suppose that
		\begin{equation}
			|\kappa_1| \geq \eps_1 \kappa_n,
		\end{equation}
		where we choose some fixed $\eps_1$ so that $0 < \eps_1 < \frac{1}{2}$.
		
		Then we have in view of the concavity of $F$, see \eqref{C1FDeriv},
		\begin{equation}
			F^{ij}h_{ik}h^k_j \geq F_1 \kappa_1^2 \geq \frac{1}{n}F^{ij}g_{ij}\eps_1^2 \kappa_n^2.
		\end{equation}
		Since $Dw=0$, 
		\begin{equation}
			D\log h^n_n = - \lambda D\chi.
		\end{equation}
		Hence 
		\begin{equation}
			F^{ij}(\log h_n^n)_i (\log h_n^n)_j \leq \lambda^2 F^{ij}\chi_i \chi_j \leq c \lambda^2 F^{ij}g_{ij}.
		\end{equation}
		For large $\kappa_n$ the first term in \eqref{C2EvWEst} is dominating, so we can conclude $\kappa_n$ is a priori bounded in this case.
		
		\begin{em} Case 2.\end{em}\quad Suppose that
		\begin{equation}
			\kappa_1 \geq - \eps_1 \kappa_n.
		\end{equation}
		Then, by using the Codazzi equations,  we can estimate the last term in \eqref{C2EvWEst} from above (where we omit the factor $\Phi'$ for a moment):
		\begin{equation}
		\begin{split}
			\frac{2}{1+\eps_1}\sum_{i=1}^n (F_n - F_i) (h_{ni;}^{\quad n})^2 (h^n_n)^{-2} \leq &\frac{2}{1+2\eps_1}\sum_{i=1}^n(F_n - F_i) (h_{nn;}^{\quad\, i})^2 (h^n_n)^{-2}\\
			&+ c(\eps_1)\sum_{i=1}^n(F_i - F_n) (h^n_n)^{-2}.
		\end{split}
		\end{equation}
		The second sum can be estimated by a constant, since $F_1 \leq c\kappa_n$.
		
		The terms in \eqref{C2EvWEst} containing the derivative of $h_n^n$ can therefore be estimated from above by (again omitting the common factor $\Phi'$)
		\begin{equation}
		\begin{split}
			&-\frac{1-2\eps_1}{1+2\eps_1}\sum_{i=1}^n F_i(h_{nn;}^{\quad\, i})^2 (h^n_n)^{-2} + \frac{2}{1+2\eps_1}F_n\sum_{i=1}^n(h_{nn;}^{\quad \, i})^2(h^n_n)^{-2}\\
			& \leq 2 F_n \sum_{i=1}^n (h_{nn;}^{\quad\, i})^2 (h^n_n)^{-2} = 2\lambda^2F_n||D\chi||^2.
		\end{split}
		\end{equation}
		Hence we get the inequality
		\begin{equation}
		\begin{split}
			0 \leq &- c F_n\kappa_n^2 + \lambda^2 c F_n + c\kappa_n + cF^{ij}g_{ij}\\
			 & + \lambda c - \lambda c_0 F^{ij}g_{ij}.
		\end{split}
		\end{equation}
		Because of $(H_2)_i \geq c_1 \kappa_n$ and the fact that $\si_2\leq c_0$ we deduce
		\begin{equation}
			F^{ij}g_{ij} \geq c \kappa_n.
		\end{equation}
		Hence we can uniformly estimate $\kappa_n$ from above, if $\lambda$ has been chosen large enough. The proposition now follows from $|A|^2 < H^2$ and $H>0$, which are valid in $\Gamma_2$ by definition.
\end{proof}
Finally we provide the $C^2$-estimates in the case of a curvature function of class $(K^*)$, see also \cite[Lemma 4.1.3]{GerhCP}.
\begin{Proposition}
		The principal curvatures of the flow \eqref{floweq} with $F \in (K^*)$, $\Phi(x) = \log(x)$, are uniformly bounded during the flow, provided there exists a strictly convex function $\chi \in C^2(\bar{\Om})$.
\end{Proposition}
\begin{proof}
	Let $\zeta$ be defined as in the preceding proof and define
	\begin{equation}
		w = \log \zeta +\lambda\tilde{v} + \mu \chi,
	\end{equation}
	where $\lambda$, $\mu$ are large positive parameters which we specify later and we will prove that $w$ is bounded if we choose $\lambda$ and $\mu$ approprietly.
	
	By the same procedure as in the last proofs we suppose $x_0$ is a point in $M(t_0)$ such that \eqref{C2SupW} holds, where $0 < T < T^*$ and $0 < t_0 \leq T$, we introduce normal coordinates at $x_0 = x(t_0, \xi_0)$, and we may define $w$ by
	\begin{equation}
		w = \log h_n^n + \lambda \tilde{v} + \mu \chi.
	\end{equation}
	At $(t_0, \xi_0)$ we have in view of the maximum principle
	\begin{equation}
	\label{C2WEst2}
	\begin{split}
		0 \leq & \,c(h_n^n + \lambda) + c\lambda F^{ij}g_{ij} + \mu c - \mu c_0 F^{ij}g_{ij}\\
		& - \lambda \Phi' F^{ij}h_i^kh_{kj} + \Phi' F^{ij}(\log h_n^n)_i (\log h_n^n)_j\\
		& + \{\Phi'' F_nF^n + \Phi' F^{kl,rs}h_{kl;n}h_{rs;}^{\quad n}\}(h_n^n)^{-1},
	\end{split}
	\end{equation}
	where we have assumed that $h_n^n$, $\lambda$ and $\mu$ are larger than $1$ and used the boundedness of $\Phi$.
	
		Since $F \in (K^*)$, by choosing $\lambda$ and $\mu$ large enough it suffices to estimate the term, which is quadratic in the derivatives. This will be done by exploiting the last term, which is negative (both its components are).
		
		Since $F \in (K)$ and $\Phi''(x) \leq -\frac{\Phi'(x)}{x}$, we can estimate the last term from above by
		\begin{equation}
		\label{C2KEst}
			- (h_n^n)^{-2} \Phi' F^{ij}h_{in;n}h_{jn;}^{\quad n}.
		\end{equation}
		The Codazzi equation implies
		\begin{equation}
			h_{in;n} = h_{nn;i} + \bar{R}_{\alpha\beta\gamma\delta}\nu^\alpha x_n^\beta x_i^\gamma x_n^\delta,
		\end{equation}
		hence by abbreviating the curvature term by $\bar{R}_i$, we conclude that \eqref{C2KEst} is equal to
		\begin{equation}
			- (h_n^n)^{-2}\Phi' F^{ij}((h_n^n)_{;i} + \bar{R}_i)((h_n^n)_{;j} + \bar{R}_j).
		\end{equation}
		Hence the last two terms in \eqref{C2WEst2} are estimated from above by
		\begin{equation}
			-2(h_n^n)^{-1}\Phi'F^{ij}(\log h_n^n)_{;i} \bar{R}_j.
		\end{equation}
		Now $Dw = 0$ yields
		\begin{equation}
			D\log h_n^n = - \lambda D\tilde{v} - \mu D\chi,
		\end{equation}
		hence we can finally estimate the last two terms by
		\begin{equation}
			\lambda c + (h_n^n)^{-1} (\lambda+\mu) c F^{ij}g_{ij}.
		\end{equation}
		This establishes the uniform bound of $\kappa_n$ from above and implies that $\kappa_1$ is uniformly bounded from below by a positive constant in view of $F \geq c > 0$ and $F_{|\partial \Gamma} = 0$.
\end{proof}

\section{Higher order estimates}
\label{higherorder}
In view of the a priori estimates obtained so far, we know that 
\begin{equation}
\label{graphC2estimates}
	|u|_{2,0, \mathcal{S}_0} \leq c_0
\end{equation}
and
\begin{equation}
	\Phi(F) \text{ is uniformly elliptic in } u
\end{equation}
independently of $t$, $0 < t < T^*$, because the principal curvatures lie in a compact subset of $\Gamma$. Denote the ellipticity constants by $\lambda, \Lambda$.

Next, we look at the nonlinear, but uniformly parabolic equation
\begin{equation}
	\frac{\partial u}{\partial t} = - e^{-\psi}v \, (\Phi(F) - f),
\end{equation}
where the operator $\Phi(F)$ is concave in $h_{ij}$, hence $-\Phi(F)$ is concave in $u_{ij}$. However, we cannot apply the Krylov-Safonov estimates since $f = f(t)$ is a merely bounded function. Instead, we can follow an argumentation similar to the one used in \cite{McCoyMixedAreaGen} and, with certain modifications, in \cite{RivSin} to obtain a uniform, time-independent bound on $u$ in $H^{2+\beta, \frac{2+\beta}{2}}([\delta, T]\times\calS_0)$ for some $0 < \beta < 1$, where we choose $\delta > 0$ to be arbitrary but fixed and $\delta < T < T^*\leq \infty$ arbitrary.

The idea is roughly as follows: 
First one obtains $H^{\beta, \frac{\beta}{2}}([\delta, T]\times \mathcal{S}_0)$-estimates, $0< \beta <\alpha$, for $u$ and $\Phi(F)$ using the parabolic Harnack-inequality. Then we fix $t \geq \delta$ and consider $u$ to be a solution of the nonlinear, but uniformly elliptic equation
\begin{equation}
	\Phi(F(\cdot, u(t, \cdot), Du(t, \cdot), D^2u(t, \cdot))) \equiv g \in C^{0, \beta}(\mathcal{S}_0).
\end{equation}
As this equation is not well-defined on the whole set $\mathcal{S}_0 \times \bbr \times \bbr^n \times \bfS$, where $\bfS \subset \bbr^{n\times n}$ denotes the space of symmetric matrices, we are going to use the Bellman-extension in a similar way as in the papers cited above. Now we can apply $C^{0,\beta}(\mathcal{S}_0)$-estimates for $D^2u(t, \cdot)$ using a result of Caffarelli (\cite[Theorem 8.1]{CafCab}). We remark that in this section $Du$ and $D^2u$ denote the first respectively second derivatives with respect to space. This then enables us to use the estimates from \cite[Sections 3.3, 3.4]{AndFN} to obtain parabolic H\"older estimates for $Du$ and $D^2u$. Finally, the higher order estimates can then be derived by using essentially the parabolic Schauder-theory, thereby asserting long-time existence.

From now on let $\eps, \delta, T$ be fixed constants with $0 < \eps < \delta < T < T^*$. By choosing a finite covering of $\mathcal{S}_0$ it suffices to show inner estimates in a fixed coordinate chart. Hence from now on all quantities of the hypersurface are expressed in local coordinates and depend on $x \in \Omega$, where $\Omega \subset\subset \bbr^n$ is an open, precompact set. 
First, we derive the H\"older-estimates for $u$ and $\Phi(F)$:
\begin{Lemma}
\label{HEFirstStep}
	There exist constants $\beta$, $0< \beta < 1$, and $c_1$ depending on the already obtained estimates, such that
		 \begin{align}
		 \label{PhiHoelder}
	 	 &||\Phi||_{\beta, \frac{\beta}{2}, \, [\eps, T]\times \calS_0} \leq c_1,\\
	 	 \label{HEbeta}
	 	 &||u||_{\beta, \frac{\beta}{2}, \, [\eps, T]\times \calS_0} \leq c_1.
	 \end{align}
\end{Lemma}
\begin{proof}
	Looking at the equation \eqref{EvPhi} we have
	\begin{equation}
	\begin{split}
		\dot{\Phi} - a^{ij} \Phi_{,ij} &+ b^k \Phi_k \equiv \dot{\Phi} - \Phi' F^{ij}\Phi_{,ij} + \Phi' F^{ij}\Gamma_{ij}^k \Phi_k \\
		&=  -\Phi' (\Phi - f)\{F^{ij}h_i^kh_{kj} + F^{ij}\bar{R}_{\alpha\beta\gamma\delta}\nu^\alpha x_i^\beta \nu^\gamma x_j^\delta\} =: g,
	\end{split}
	\end{equation}
	and $g$ as well as $b^k$ are bounded in view of \eqref{graphC2estimates}. Since this is a uniformly parabolic equation, we can apply \cite[Corollary 7.41]{Lieberman} or \cite[Theorem 3.16]{SchnuererPde} to obtain \eqref{PhiHoelder}.
	Next we look at \eqref{EvParGraph}:
	\begin{align}
			 \dot{u} - \Phi' F^{ij}u_{,ij} &+ \Phi'F^{ij}\Gamma_{ij}^k u_k = -e^{-\psi}\tilde{v}(\Phi - f) + \Phi'Fe^{-\psi}\tilde{v}\\
	 & + \Phi' F^{ij}\{\bar{\Gamma}^0_{00}u_iu_j + 2 \bar{\Gamma}^0_{0i}u_j + \bar{\Gamma}^0_{ij}\}\notag
	\end{align}
	Since we are in the same situation as before, proceeding as above yields \eqref{HEbeta}.
\end{proof}

Next, we want to use the result of Caffarelli to obtain spatial $C^{2,\beta}$ estimates for $u$. First we extend $F$ using the Bellman-extension to all of $\bbr^{n\times n}$, where we consider $F = F(h^i_j)$:
\newpage
\begin{Lemma}
\label{FExtension}
	Let $\Gamma \subset \bbr^n$ be an open, convex, symmetric cone and $f \in C^2(\Gamma)$ be a symmetric, concave function, positively homogeneous of degree 1 and vanishing on the boundary of $\Gamma$. Denote by $F \in C^2(\mathbf{S}_\Gamma)$ the corresponding curvature function. Let $C_1, C_2$ be positive constants. Then there exists a function $\tilde{F}$ defined on all of $\bbr^{n\times n}$, which agrees with $F$ on the set
	\begin{equation}
		C = \left\{ A\in \mathbf{S}_\Gamma: C_1 \leq F \, \wedge \, \underset{1\leq i \leq n}{\max}\, \kappa_i \leq C_2\right\},
	\end{equation}
	where the $\kappa_i$ denote the eigenvalues of $A$. Furthermore $\tilde{F}$ is uniformly Lipschitz continuous, positively homogeneous of degree 1, concave and uniformly elliptic, where the ellipticity constants depend on $C_1$ and $C_2$.
\end{Lemma}
\begin{proof}
	Define $\tilde{F}$ for arbitrary $(b_j^i) \in \bbr^{n\times n}$ as follows:
	\begin{equation}
		\tilde{F}(b_j^i) = \underset{(h_j^i) \in C}{\min}\, F^k_l(h_j^i) \, b^l_k.
	\end{equation}
	First of all, in view of Eulers homogeneity relation $F^k_l(h_j^i) h^l_k = F(h_j^i)$ and the concavity we infer for $(b_j^i), (h_j^i) \in \Gamma$ 
	\begin{equation}
		F(b_j^i) \leq F(h_j^i) + F^k_l(h_j^i)(b_k^l - h_k^l) = F^k_l(h_j^i) \, b_k^l.
	\end{equation}
	From this inequality and again the homogeneity relation we infer $\tilde{F}_{|C} = F_{|C}$. By definition $\tilde{F}$ is concave, homogeneous of degree 1 and well-defined, since $C$ is compact. Furthermore, using some elementary properties of the trace-function and the fact that $F$ is uniformly monotone in $C$, one obtains the uniform ellipticity of $\tilde{F}$, hence there exist positive constants $\lambda$, $\Lambda$, depending on $n$, $F$ and $C$, such that $\forall A, B \in \bbr^{n\times n}$ with $B$ nonnegative definite there holds
	\begin{equation}
		\lambda\, ||B|| \leq \tilde{F}(A + B) - \tilde{F}(A) \leq \Lambda \, ||B||.
	\end{equation}
	In the same way as above one can establish that $\tilde{F}$ is uniformly Lipschitz continuous.
\end{proof}

Now we can derive the spatial $C^{2,\beta}$-estimates, first we cite the result of Caffarelli, see \cite[Theorem 8.1]{CafCab} and the remarks following it:
\begin{Theorem}
\label{NonlinearCaffarelli}
	Let  $g \in C^{0, \alpha}(\Omega)$, $0<\alpha < 1$, $G: \Om \times \bfS \rightarrow \bbr$ be continuous, concave in the second argument, uniformly elliptic, i.e. there are constants $\mu_1, \mu_2$ such that for all $x\in \Om$, $A, B \in \bfS$ with $B$ nonnegative definite there holds
	\begin{equation}
		\mu_1 ||B|| \leq G(x, A + B) - G(x, A) \leq \mu_2 ||B||
	\end{equation}
	and furthermore there exists $c >0$, such that for all $x, y \in \Om$, $A\in \bfS$ there holds
	\begin{equation}
	\label{NonlinearGHom}
		|G(x, A) - G(y, A)| \leq c |x-y|^\alpha \, (||A|| +1).
	\end{equation}
	Then for $\Om' \subset \subset \Om$ there exist constants $0<\beta < \alpha$ and $C > 0$ such that a solution $u$ of $G(\cdot, D^2u(\cdot)) = g(\cdot)$ satisfies
	\begin{equation}
		||u||_{2, \beta, \Om'} \leq C \,(||u||_{0, \Om} + ||g||_{\alpha, \Om} + 1),
	\end{equation}
	where $C$ depends on $n, \mu_1, \mu_2, c$ and $\Om'$. 
\end{Theorem}

\begin{Lemma}
\label{etaDefined}
Let $\Omega' \subset \subset \Omega$ be an open set and $t \in [\eps, T]$ arbitrary, then there exist constants $\tilde{\beta}$, $0<\tilde{\beta} < \beta$, and $c_2$ such that
\begin{equation}
	||D^2u(t, \cdot)||_{\tilde{\beta}, \Om'} \leq c_2,
\end{equation}
where $c_2$ is a constant depending on $\lambda, \Lambda, d(\Om', \partial \Om)$ and the constants $c_0$ from \eqref{graphC2estimates} and $c_1$ from \eqref{PhiHoelder}.
\end{Lemma}
\begin{proof}
	By using \eqref{EvHUlin} we can define a smooth function
	\begin{equation}
	\begin{split}
		&\eta: \Om \times \bbr \times B_1^n(0) \times \bfS \rightarrow \bfS, \\
		&\eta = \eta(x, z, p, r) = \left(-\chi^{ik}(x, z, p) r_{kj} + R_j^i(x, z, p)\right), 
	\end{split}
	\end{equation}
	linear in $r$ for fixed $x, z, p$, such that for arbitrary spacelike hypersurfaces $M = $ graph $u_{|\Om}$ of class $C^2$ we have for $x \in \Om$
	\begin{equation}
		h_j^i(u(x), x) = [\eta(x, u(x), Du(x), D^2u(x))]_j^i,
	\end{equation}
	where the derivatives are partial derivatives. Furthermore for such $u$ we can define the function
	\begin{equation}
	\begin{split}
		\bar{\eta}_{u}:& \Om \times \bfS \rightarrow \bfS,\\
		& (x, r) \mapsto \eta(x, u(x), Du(x), r),
	\end{split}
	\end{equation}
	which is now of class $C^1$.
	
	Applying Lemma \ref{FExtension}, with $C_1$ and $C_2$ chosen correspondingly to the already obtained a priori estimates, we define
	\begin{equation}
	\begin{split}
		G: &\Om \times \bfS \times (0, T] \rightarrow \bbr,\\
		&(x,r,t) \mapsto \Phi\big(\tilde{F}(\bar{\eta}_{u(t)}(x, r))\big).
	\end{split}
	\end{equation}
	For fixed $t \in [\eps, T]$ and $g(\cdot) := -G(\cdot, D^2u(t, \cdot), t) \in C^{0, \beta}(\Om)$, which is valid in view of Lemma \ref{HEFirstStep}, we consider $u(t, \cdot)$ to be a solution to the equation
	\begin{equation}
		 -G(x, D^2v(x), t) = g(x)\quad , \, x \in \Om, \, v\in C^2(\Om).
	\end{equation}
	Now we can apply Theorem \ref{NonlinearCaffarelli} to obtain the desired estimates, where we use the uniform ellipticity of $\tilde{F}$ and we note that $-G$ is concave with respect to $r$ (see the definition of $\eta$ and the remark at the beginning of this section). Finally, a short computation using the already obtained $C^2$-estimates yields that \eqref{NonlinearGHom} is satisfied (even for $\alpha = 1$).

\end{proof}

Finally, we obtain:
\begin{Proposition}
\label{C2betaEstimates}
	Fix $\delta$, $T$, such that $0<\delta< T < T^*\leq \infty$. Then for the solution of problem \eqref{floweq} we have uniform estimates
	\begin{equation}
		||u||_{2+\beta, \frac{2+\beta}{2}, \,[\delta, T]\times \calS_0} \leq c
	\end{equation}
	for constants $\beta$, $0 < \beta < 1$, and $c$ depending on the choice of $\delta$, but not on $T$.
\end{Proposition}
\begin{proof}
	In the same way as we obtained a uniformly elliptic equation for $u$ in the last Lemma, we can obtain a uniformly parabolic equation for $u$ and use the estimates from \cite[Sections 3.3, 3.4]{AndFN} to obtain the parabolic H\"older-estimates for $Du$ and $D^2u$. Together with \eqref{HEbeta} and the definition of $f$ we obtain (omitting the tilde in the H\"older exponential)
	\begin{equation}
		||f||_{\frac{\beta}{2}, [\delta, T]} \leq \textnormal{const}.
	\end{equation}
	The last estimate to complete the Proposition can then be obtained from the evolution equation of the graph.
\end{proof}

We finish this section with the higher order estimates and the existence for all times.
\begin{Proposition}
	The scalar curvature flow exists for all times $0 < t < \infty$ in the class $H^{m+2+\alpha, \frac{m+2+\alpha}{2}}([0,t] \times \calS_0)$ and the curvature flow exists for all times in the class $H^{m+\alpha, \frac{m+\alpha}{2}}(Q_t, N)$. Furthermore we have uniform estimates for the scalar curvature flow, i.e. there exists a constant $c>0$, such that
	\begin{equation}
	\label{HEEstimates}
		||u||_{m+2+\alpha, \frac{m+2+\alpha}{2}, [0, \infty)\times \calS_0} \leq c.
	\end{equation} 
\end{Proposition}
\begin{proof}
	Going through the proof of \cite[Theorem 2.5.9]{GerhCP} reveals that the time dependence of $f$ does not cause any problems, as it already has the right regularity to proceed. Following the arguments in \cite[2.5.12]{GerhCP} and noting that now $\Psi$ has again the same regularity as in that proof due to the definition of $f$, we conclude that \cite[Lemma 2.5.17]{GerhCP} is also valid for the volume preserving flows. Hence in view of the a priori estimates from Proposition \ref{C2betaEstimates} we infer that the flow exists for all times and has the regularity mentioned above, compare the argumentation in \cite[Remark 2.6.2]{GerhCP}. 
	
	It remains to prove the estimate \eqref{HEEstimates}. With $\beta$ from Proposition \ref{C2betaEstimates}, we have $u \in H^{m+2+\beta, \frac{m+2+\beta}{2}}([0,\infty) \times \calS_0)$, see \cite[Theorem 6.5]{GerhSurvey}. With this new a priori estimates at hand we can again apply the Theorem to obtain the uniform estimates \eqref{HEEstimates}.
\end{proof}
	
	We remark that the estimates we have used so far, apart from Lemma \ref{FBoundsLemma}, do not rely too much on the particular choice of the global force term. Hence as long as bounds for the global force term can be established, one can alter for example the integrands by functions, which depend on $u$ up to its second derivatives and still obtain long time existence for the flow.

\section{Convergence}
\label{Convergence}
Now we want to show the convergence to a hypersurface of constant $F$-curvature. The first step consists of proving the convergence of the $F$-curvature. Let us first cite a well-known fact.

\begin{Lemma}
\label{LemmaLipschitz}
	Let $\mathcal{S}_0$ be a compact manifold of class $C^1$ and $f \in C^1(J\times \mathcal{S}_0)$, where $J$ is an open interval, then 
	\begin{equation}
		\phi(t) = \underset{\mathcal{S}_0}{\sup }\, f(t, \cdot)
	\end{equation}
	is Lipschitz continuous and there holds a.e.
	\begin{equation}
		\dot{\phi}(t) = \frac{\partial f}{\partial t}(t, x_t),
	\end{equation}
	where $x_t$ is a point in which the supremum is attained.
	
	A corresponding result is also valid if $\phi$ is defined by taking the infimum instead of the supremum.
\end{Lemma}
\begin{proof}
	See Lemma 6.3.2 in \cite{GerhCP}.
\end{proof}

%
%
%
%

Dealing carefully with the equation \eqref{EvPhi} we can prove at once the exponential convergence of the $F$-curvature. Note that the proof does not rely on any a priori estimates besides the bounds on the curvature function.

\begin{Lemma}
\label{Fexp}
	There exist constants $0 < \delta = \delta(M_0)$ and $c = c(M_0)$, such that
	\begin{equation}
		\underset{x \in M_t}{\sup}|\Phi(F)(x) - f_k(t)| \leq c\, e^{-\delta t}.
	\end{equation}
\end{Lemma}
\begin{proof}
We remind that $\Phi(F)$ satisfies a parabolic equation of the form
\begin{equation}
\label{PhiFlowSimple}
	\dot{\Phi} - a^{ij}\Phi_{;ij} + C(\Phi - f_k) = 0
\end{equation}
with
\begin{equation}
	C := \Phi' \{F^{ij}h_i^kh_{kj} + F^{ij}\bar{R}_{\alpha\beta\gamma\delta}\nu^\alpha x_i^\beta \nu^\gamma x_j^\delta\}
\end{equation}
and we note $C \geq c_0 > 0$ and $c_0 = c_0(M_0)$, in view of Lemma \ref{FBoundsLemma}.

We consider $\Phi$ as a function 
\begin{equation}
\begin{split}
	\Phi: &[0, T^*) \times M \rightarrow \bbr\\
	&(t, \xi) \mapsto \Phi(F((u(t, \xi), x^i(t, \xi)))
\end{split}
\end{equation} 
and denote by $\xi^{\inf}(t) \in M_t$ a point where
\begin{equation}
	\Phi^{\inf}(t) =	\Phi(\xi^{\inf}(t)) := \underset{\xi \in M}{\text{inf }} \Phi(t, \xi) 
\end{equation} 
and by $\xi^{\sup}(t) \in M_t$ a point where
\begin{equation}
	\Phi^{\sup}(t) =	\Phi(\xi^{\sup}(t)) :=\underset{\xi \in M}{\text{sup }} \Phi(t, \xi).
\end{equation} 
Let 
\begin{equation}
	\eta(t) := \Phi^{\sup}(t) - \Phi^{\inf}(t).
\end{equation}
We know that $\Phi^{\inf}$ and $\Phi^{\sup}$ are lipschitz continuous considered as functions depending on $t$, hence by the previous Lemma there holds for a.e. $t$:
\begin{equation}
\begin{split}
	0 &= \dot{\eta} - (\Phi'F^{ij}\Phi_{;ij})(\xi^{\sup}) + (\Phi' F^{ij}\Phi_{;ij})(\xi^{\inf})\\
	 &+ (\Phi^{\sup} - f_k)C(\xi^{\sup}) - (\Phi^{\inf} - f_k)C(\xi^{\inf}).
\end{split}
\end{equation}
Considering the points at which the functions are evaluated, one obtains the following inequality
\begin{equation}
	0 \geq \dot{\eta} + (\Phi^{\sup} - f_k)C(\xi^{\sup}) - (\Phi^{\inf} - f_k)C(\xi^{\inf}).
\end{equation}
Now since both $\Phi^{\sup} - f_k$ and $f_k - \Phi^{\inf}$ are nonnegative, due to the definition of $f_k$, we conclude
\begin{equation}
	0 \geq \dot{\eta} + c_0 (\Phi^{\sup} -f_k + f_k - \Phi^{\inf})= \dot{\eta} + c_0 \eta.
\end{equation}
Hence there holds
\begin{equation}
	0 \geq \frac{d}{dt}\left(e^{c_0t}\eta\right)
\end{equation}
for a.e. $t \in [0, T^*)$. Integrating over $t$ shows the exponential decay and proves the lemma.

\end{proof}

Now we can infer the convergence of the graphs:
\begin{Korollar}
\label{graphConv}
	The graphs $u = u(t)$ converge exponentially to a continuous function $u_\infty$ on $\mathcal{S}_0$ in the Supremum-Norm, where the factor in the exponential convergence is the same as in Lemma \ref{Fexp}, i.e. there exists a constant $\bar{c} = \bar{c}(M_0, |u|) > 0$ such that
	\begin{equation}
		\underset{x \in \mathcal{S}_0}{\sup} |u(t, x) - u_\infty(x)| \leq \bar{c}\, e^{-\delta t}.
	\end{equation}
\end{Korollar}
\begin{proof}
	Let $t \in [0,\infty)$ be given and $t' > t$ be arbitrary. Then we have in view of \eqref{partialU} for an arbitrary $x \in \mathcal{S}_0$ and some $c' > 0$:
	\begin{equation}
		|u(t, x) - u(t', x)| \leq \frac{c'}{\delta} e^{-\delta t}.
	\end{equation}
\end{proof}

We remind a well-known interpolation Lemma, which will be used to show the exponential convergence of the graphs in $C^{m+2}$.
\begin{Lemma}
	Let $\Om$ be a bounded open subset of $\bbr^n$ and $\Om' \subset \subset \Om$ be an open subset. Furthermore let $m, l \in \bbn$, $1 \leq l < m$, $\alpha \in \bbr$, $0 < \alpha \leq 1$. Then the following two interpolation inequalities are valid:
	\begin{enumerate}[(i)]
	\item There exists $c>0$, where $c=c(n, m, \Om')$, such that for all $u \in C^m(\bar{\Om})$ there holds
	\begin{equation}
	\label{interpolCm}
		||u||_{l, \Om'} \leq c \, ||u||_{0, \Om}^{\frac{m-l}{m}} \, (||u||_{0, \Om}^{\frac{l}{m}} +||D^mu||_{0, \Om}^{\frac{l}{m}}).
	\end{equation}
	\item There exists $c>0$, where $c= c(n, m, \alpha, \Om')$, such that for all $u \in C^{m, \alpha}(\bar{\Om})$ there holds
	\begin{equation}
	\label{interpolCmalpha}
		||u||_{m, \Om'} \leq c \, ||u||_{0, \Om}^{\frac{\alpha}{m+\alpha}} \,(||u||_{0,\Om}^{\frac{m}{m+\alpha}} + [D^mu]_{\alpha, \Om}^{\frac{m}{m+\alpha}}).
	\end{equation}
	\end{enumerate}\vspace{0.2cm}
\end{Lemma}

From the preceding Lemmata one can infer the exponential convergence in $C^{m+2}(\mathcal{S}_0)$:
\begin{Korollar}
	The functions $u(t,\cdot)$ converge exponentially for $t \to \infty$ in $C^{m+2}(\mathcal{S}_0)$ to $u_\infty \in C^{m+2,\alpha}(\mathcal{S}_0)$. $u_\infty$ represents a spacelike hypersurface of class $C^{m+2, \alpha}$ with constant $F$-curvature.
\end{Korollar}
\begin{proof}
	Using the uniform estimates \eqref{HEEstimates} together with Corollary \ref{graphConv} and the interpolation inequality \eqref{interpolCmalpha} we conclude the exponential convergence of $u(t,\cdot)$ in $C^{m+2}(\calS_0)$. Since we have uniform estimates for $\tilde{v}$, the limit hypersurface $M_\infty = $ graph $u_\infty$ is a spacelike hypersurface. Lemma \ref{Fexp} shows that the limit hypersurface has constant $F$-curvature, then the elliptic Schauder theory implies $u_\infty \in C^{m+2,\alpha}(\mathcal{S}_0)$.
\end{proof}

If we assume the initial hypersurface and the considered curvature function to be smooth, then the above Lemma yields the exponential convergence in the $C^\infty$-topology:

\begin{Korollar}
	If the initial hypersurface and the curvature function $F$ are smooth, then the graphs converge exponentially in the $C^\infty$-Topology to a hypersurface of constant $F$-curvature.
\end{Korollar}

\section{Stability}
In this section we want to prove the strict stability of the limit hypersurface, which means, that for curvature functions of class $(D)$ the first eigenvalue of the linearization is strictly positive.

First, we linearize the operator $F$. For this let $M_0$ be a hypersurface, which satisfies
\begin{equation}
\label{FequalsConst}
	F_{|M_0} = c,
\end{equation}
where $c$ is a constant (positive in case $F$ is of class $(K^*)$ and arbitrary for $F=H$). Then there holds, see \cite[Lemma 3.9]{GerhSurvey}:
\begin{Lemma}
\label{Linearization}
	Let $M_0$ be of class $C^{m+2,\alpha}$, $m \geq 2$, $0 \leq \alpha \leq 1$, and satisfy \eqref{FequalsConst}. Let $\mathcal{U}$ be a tubular neighbourhood of $M_0$, then the linearization of the operator $F$ expressed in the normal Gaussian coordinate system $(x^\alpha)$ corresponding to $\mathcal{U}$ and evaluated at $M_0$ has the form
	\begin{equation}
	\label{linearizationOperator}
		Bu := -F^{ij}u_{ij} + \{F^{ij}h_i^kh_{kj} + F^{ij}\bar{R}_{\alpha\beta\gamma\delta}\nu^\alpha x_i^\beta \nu^\gamma x_j^\delta\}u,
	\end{equation}
	where $u$ is a function on $M_0$ and all geometric quantities are those of $M_0$. The derivatives are covariant derivatives with respect to the induced metric of $M_0$. The operator is self-adjoint, if $F^{ij}$ is divergence free.
\end{Lemma}
We remind the definition of stability:
\begin{Definition}
	Let $N$ be Lorentzian, $F$ a curvature operator, and $M \subset N$ a compact, spacelike hypersurface, such that $M$ is admissible. Then $M$ is said to be a (strictly) stable solution to the equation \eqref{FequalsConst}, if the quadratic form
	\begin{equation}
		\int_M{F^{ij}u_i u_j} + \int_M{\{F^{ij}h_{ik}h^k_j + F^{ij}\bar{R}_{\alpha\beta\gamma\delta}\nu^\alpha x_i^\beta \nu^\gamma x_j^\delta\}\} u^2}
	\end{equation}
	is (positive) non-negative for all $u\in C^2(M)$, $u \not\equiv 0$. If $F$ is of class $(D)$, i.e. $F^{ij}$ is divergence free, then this is equivalent to the fact, that the first eigenvalue $\lambda_1$ of the linearization, which is the operator in \eqref{linearizationOperator}, is non-negative.
\end{Definition}
In view of the assumptions on the ambient manifold $N$, in our case there holds
\begin{Proposition}
	The limit hypersurface of the flow is strictly stable.
\end{Proposition}

\section{Foliation}
In this section we want to derive some results for regions covered by constant $F$-curvature surfaces, but first we are going to show that under suitable assumptions we can provide such a foliation. To show the existence of a region covered by compact, connected, spacelike constant $F$-curvature hypersurfaces (such a hypersurface will be called CFC-surface from now on) we use however the corresponding curvature flow with the volume preserving term substituted by a constant. The corresponding results can be found in \cite[Theorem 4.2.1,Theorem 5.1.1 and Theorem 4.1.1]{GerhCP}, respectively for $H$, $\si_2$ and $F \in (K^*)$. For the convenience of the reader we state the results from this Theorems:
\begin{Theorem}
\label{ExistenceCFC}
Let $N$, $F$, $\Gamma$ be as in section \ref{Introduction} with $m\geq 2$, $0 < \alpha <1$. If $c > 0$ is a constant and there exists a future and a past curvature barrier for $(F, \Gamma, c)$ of class $C^{m+2,\alpha}$, then there exists a compact, connected, spacelike hypersurface $M$ of class $C^{m+2, \alpha}$ satisfying the equation
\begin{equation}
	F_{|M} = c,
\end{equation}
provided there exists a strictly convex function $\chi \in C^2(\bar{\Om})$, where $\Om$ is the region between the barriers. In the case $F=H$ we do not need the existence of the strictly convex function.
\end{Theorem}
Using this theorem we can show the existence of a foliation in a region enclosed by barriers by following the arguments used to establish a foliation by constant mean curvature surfaces in \cite[Theorem 4.6.3]{GerhCP}.
\begin{Theorem}
\label{FoliationThm}
	Let $N$, $F$, $\Gamma$ be as in Theorem \ref{MainTheorem1} with $m\geq 2$, $0 < \alpha < 1$. Let $c_1 < c_2$ be positive constants and suppose there exists a future curvature barrier for $(F, \Gamma, c_2)$ and a past curvature barrier for $(F, \Gamma, c_1)$, both of class $C^{m+2,\alpha}$, and denote the region between the barriers by $\Omega$. If $F$ is not the mean curvature, then we suppose in addition that there exists a strictly convex function $\chi \in C^2(\bar{\Om})$. Let $M_{c_1}$, $M_{c_2}$ be the CFC-surfaces with $F$-curvature equal to $c_1$ respectively $c_2$. Then the region between $M_{c_1}$ and $M_{c_2}$, which will be denoted by $N_0$, can be foliated by CFC-surfaces of class $C^{m+2,\alpha}$ and there exists a time function $x^0$ of class $C^{m-1}$, such that the slices
\begin{equation}
	M_\tau = \{x^0 = \tau\}, \quad c_1 < \tau < c_2,
\end{equation} 
have $F$-curvature $\tau$. 
\end{Theorem}
First we show the existence of the foliation, this is done in the following
\begin{Lemma}
\label{LemmaFoliation}
	Under the assumptions of Theorem \ref{FoliationThm}, there exist CFC-surfaces $M_\tau$ of class $C^{m+2,\alpha}$ for each $c_1 \leq \tau \leq c_2$ such that
	\begin{equation}
	\label{ExistenceFoliation}
		\bar{N_0} = \bigcup_{c_1 \leq \tau \leq c_2} M_\tau.
	\end{equation}
	Furthermore the $M_\tau$ can be written as graphs over $\calS_0$
	\begin{equation}
		M_\tau = \textnormal{graph } u(\tau, \cdot),
	\end{equation}
	such that $u$ is strictly monotone increasing with respect to $\tau$ and continuous in $[c_1, c_2]\times \calS_0$.
\end{Lemma}
\begin{proof}
	This follows as in \cite[Lemma 4.6.2]{GerhCP} by using Theorem \ref{ExistenceCFC}, the uniqueness of $CFC$-surfaces and the monotonicity of $F$ for level hypersurfaces in a tubular neighbourhood around a fixed $CFC$-surface.
\end{proof}
Now we can prove Theorem \ref{FoliationThm}:
\begin{proof}
	We have to show that the $F$-curvature parameter can be used as a time function, i.e., $\tau$ should be of class $C^{m-1}$ with non-vanishing gradient.
	
	The regularity of $\tau$ can be shown in an arbitrary coordinate system and it suffices to prove it locally.
	Let $\tau' \in (c_1, c_2)$ and consider a tubular neighbourhood $\mathcal{U} = (-\delta, \delta) \times M_{\tau'}$ with $\delta > 0$ around $M_{\tau'}$ and the corresponding normal gaussian coordinate system of class $C^{m+1, \alpha}$, see \cite[Theorem 12.5.13]{GerhAna}. Then for small $\eps > 0$ we have
	\begin{equation}
		M_\tau \subset \mathcal{U} \quad \forall \, \tau \in (\tau' - \eps, \tau' + \eps), 
	\end{equation}
	see the proof of the Lemma above, they can be written as graphs over $M_{\tau'}$, $M_\tau = $ graph $u(\tau, \cdot)$ and using the implicit function theorem we will show that $u$ is of class $C^{m-1}$:
	
	Let $\delta > 0$ and $s \in \bbn$, $0\leq s \leq m-2$, then we define the open subset 
	\begin{equation}
		C_\delta^s := \{\varphi \in C^{s+2, \alpha}(M_{\tau'}): ||\varphi||_{s+2,\alpha, M_{\tau'}} < \delta\}
	\end{equation} 
	of the Banach space $C^{s+2, \alpha}(M_{\tau'})$, which is equipped with a norm induced by the induced metric of $M_{\tau'}$. If $\varphi \in C_\delta^s$ and $\delta$ is sufficiently small, then graph $\varphi$ represents a compact, connected, spacelike and admissible hypersurface, hence we can define the operator
	\begin{equation}
	\begin{split}
		&G^s: (\tau' - \eps, \tau' + \eps) \times C_\delta^s \rightarrow C^{s, \alpha}(M_{\tau'}), \\
		&G^s(\tau, \varphi) = F(\varphi) - \tau,
	\end{split}
	\end{equation}
	where $F(\varphi)$ denotes the $F$-curvature of graph $\varphi_{|M_{\tau'}}$.
	
	We will show now that $G^s$ is of class $C^{m-s-1}$, since $F$ is of class $C^{m}$. We want to express the operator $F : C_\delta^s \rightarrow C^{s, \alpha}(M_{\tau'})$ as a composition of several mappings, for which we can prove the regularity needed, especially we want to be in a position to use Lemma \ref{SuperpositionOp} below, i.e. we want to localize the operator $F$. From now on let $s$ be fixed.
	
	First of all, let $(\tilde{U}_i, \varphi_i)_{1 \leq i \leq k}$, $k\in \bbn$, be a covering of $M_{\tau'}$ by coordinate charts $\varphi_i: \tilde{U}_i \rightarrow \tilde{\Om}_i$, $\tilde{\Om}_i \subset \bbr^n$ open, such that there exist open, precompact subsets $\Om_i \subset \subset \tilde{\Om}_i$ satisfying $\bigcup_{i=1}^k \varphi_i^{-1}(\Om_i) = M_{\tau'}$. Let $\bar{U}_i := \varphi_i^{-1}(\bar{\Om}_i)$. Then define the linear and continuous, and hence smooth, mapping
	\begin{equation}
	\begin{split}
		\psi: &C^{s+2, \alpha}(M_{\tau'}) \rightarrow \prod_{i=1}^k C^{s+2, \alpha}(\bar{\Om}_i),\\
		& u \mapsto (u\circ \varphi^{-1}_{1|\bar{\Om}_1}, \ldots, u\circ \varphi^{-1}_{k|\bar{\Om}_k}).
	\end{split}
	\end{equation}
	
	Next, for $1 \leq i \leq k$ we define the linear and continuous, hence again smooth, mappings
	\begin{equation}
	\begin{split}
		\gamma^i: &C^{s+2, \alpha}(\bar{\Om}_i) \rightarrow C^{s, \alpha}(\bar{\Om}_i, \bbr \times \bbr^n \times \mathbf{S})\\
		& u \mapsto (u, Du, D^2u),
	\end{split}
	\end{equation}
	where $\mathbf{S}$ denotes the symmetric $n\times n$-matrices and the derivatives are partial derivatives. Denote by $\gamma$ the map with components $\gamma^i$.
	
	We denote by $\eta^i$, $1\leq i \leq k$, the function $\eta$ from Lemma $\ref{etaDefined}$ defined on the corresponding set $\tilde{\Om}_i$, thus it is the function representing the second fundamental form for graphs over $M_{\tau'}$ in the coordinate chart $(\tilde{U}_i, \varphi_i)$. We note that $\eta^i$ is of class $C^{m, \alpha}$, since it can be shown, by going through the proof of the tubular neighbourhood theorem, that the Christoffel-symbols appearing in \eqref{EvHUlin} through the equation $\bar{h}_{ij} = -\bar{\Gamma}^0_{ij}$ are of class $C^{m, \alpha}$ in a tubular neighborhood of a hypersurface of class $C^{m+2, \alpha}$. We restrict $\eta^i$ to the open set $\tilde{\Om}_i \times X_i$, on which the $F$-curvature is well defined (preimage of the open cone of definition) and define
	\begin{equation}
	\begin{split}
		F^i: &\bar{\Om}_i \times X_i \rightarrow \bbr,\\
		& (x, z, p, r) \mapsto F(\eta^i(x, z, p, r)).
	\end{split}
	\end{equation}
	Let $B^{s+2}_i := C^{s+2, \alpha}(\bar{\Om}_i, X_i)$ and denote by $B^{s+2} \subset \prod_{i=1}^k C^{s+2, \alpha}(\bar{\Om}_i, \bbr \times \bbr^n \times \bfS)$ the open subset with components belonging to $B^{s+2}_i$. Now we can apply Lemma \ref{SuperpositionOp} to obtain that that the induced maps $\tilde{F}^i: B^{s+2}_i \rightarrow C^{s, \alpha}(\bar{\Om}_i)$ are of class $C^{m-s-1}$. It remains to put these maps together to obtain the $F$-curvature of graph $u$ defined on $M_{\tau'}$:
	
	Let $(\zeta_i)_{1\leq i \leq k}$ be a partition of unity subordinate to the covering $(U_i)_{1\leq i\leq k}$, and define
	\begin{equation}
	\begin{split}
		\Phi: &B^{s+2} \rightarrow C^{s, \alpha}(M_{\tau'}),\\
		&(u_1, \ldots, u_k) \mapsto \sum_{i=1}^k{ \tilde{F}^i (u_i \circ \varphi_{i |U_i}) \cdot \zeta_i}.
	\end{split}
	\end{equation}
	As can be seen by an argumentation as in the previous steps, this map is of class $C^{m-s-1}$ and $F$ as a map from $C_\delta^s$ to $C^{s, \alpha}(M_{\tau'})$ equals $\Phi \circ \gamma \circ \psi$, hence it is also of class $C^{m-s-1}$, completing this part of the proof.
	
	Now Lemma \ref{Linearization} implies 
	\begin{equation}
		D_2G^s(\tau', 0) \varphi = - F^{ij} \varphi_{ij} + \{F^l_k h_l^m h_m^k + F_k^l \bar{R}_{\alpha\beta\gamma\delta}\nu^\alpha x_l^\beta \nu^\gamma x_m^\delta g^{mk}\} \varphi,
	\end{equation}
	where the geometric quantities appearing in this equation correspond to $M_{\tau'}$. Hence the elliptic Schauder theory implies that the operator
	\begin{equation}
		D_2G^s(\tau', 0): C^{s+2, \alpha}(M_{\tau'}) \rightarrow C^{s, \alpha}(M_{\tau'})
	\end{equation}
	is an isomorphism and the implicit function theorem implies the existence of $\hat{u}^s \in C^{m-s-1}((\tau'-\gamma^s, \tau' + \gamma^s), C^{s+2, \alpha}(M_{\tau'}))$ for some small $\gamma^s > 0$, such that $G^s(\tau, \hat{u}^s(\tau, \cdot))= 0$. Let $\gamma := \underset{0\leq s \leq m-2}{\min}\, \gamma^s$. 
	
	We will show the regularity of $u$ in a coordinate chart $(\Om, \phi)$ of $M_{\tau'}$, where $\phi$ is of class $C^{m+2, \alpha}$, $\Om \subset \subset M_{\tau'}$ is a domain and let $\Om' \subset \phi(\Om)$ be a domain with a smooth boundary. Then we can define  
	\begin{equation}
	\begin{split}
		\bar{u}^s: &(\tau' - \gamma, \tau' + \gamma) \rightarrow C^{s+2, \alpha}(\bar{\Om}'),\\
		&t \mapsto \hat{u}^s(t) \circ (\phi^{-1})_{|\bar{\Om}'},
	\end{split}
	\end{equation}
	which is then again of class $C^{m-s-1}$.
	
	Furthermore for $0 \leq s \leq m-2$ we define the supplementary function
	\begin{equation}
	\begin{split}
		\chi^{s+2}: &\bar{\Om}' \rightarrow L(C^{s+2, \alpha}(\bar{\Om}'), \bbr),\\
		&x \mapsto \left( \chi^{s+2}(x): u\mapsto u(x) \right).
	\end{split}
	\end{equation}
	Then $\chi^{s+2}$ is of class $C^{s+2, \alpha}$, where for a $n$-dimensional multi-index $\beta$ with $|\beta|\leq s+2$ there holds $D^\beta \chi = \eta^{s+2, \beta}$, which is defined as 
	\begin{equation}
	\begin{split}
		\eta^{s+2, \beta}:&\bar{\Om}' \rightarrow L(C^{s+2, \alpha}(\bar{\Om}'), \bbr),\\
		&x \mapsto \left( \eta^{s+2, \beta}(x): u \mapsto D^\beta u(x)\right).
  \end{split}
	\end{equation}
	Finally, we consider the function
	\begin{equation}
	\begin{split}
		u: &(\tau'-\gamma, \tau'+\gamma) \times \bar{\Om}' \rightarrow \bbr,\\
		&(\tau, x) \mapsto \chi^{s+2}(x)\bar{u}^s(\tau),
	\end{split}
	\end{equation}
	which is well defined independently of $s$ in view of the uniqueness of CFC-surfaces. Now let $\beta$ be an $n+1$-dimensional multi-index with $|\beta|\leq m-1$ and denote by $\hat{\beta}$ the last $n$ components of $\beta$. To be precise, at this moment we should also include an order of the elements of $\beta$, which would correspond to the order of the partial derivatives to be taken, however the proof below still holds unchanged for ordered multi-indices. If $\beta_1 > 0$ then define $s:= m - 1 - \beta_1$ and for $\beta_1 =0$ define $s := m-2$.
	Then $D^\beta u(t, x)$ exists and using the chain rule we see that $D^\beta u(t, x) = D^{\hat{\beta}}\chi^{s+2}(x) \circ D^{\beta_1}\bar{u}^s(t)$ and hence is continuous.
	We conclude that $u \in C^{m-1}((\tau'-\gamma, \tau'+\gamma) \times M_{\tau'})$.
		 	
	Next we show that $\tau$ has a non-vanishing gradient:
	Again in a tubular neighbourhood of $M_{\tau'}$ we define the coordinate transformation
	\begin{equation}
		\Phi(\tau, x^i) = (u(\tau, x^i), x^i).
	\end{equation} 
	Then there holds
	\begin{equation}
		\det D\Phi = \frac{\partial u}{\partial \tau} = \dot{u}.
	\end{equation}
	If we can show that $\dot{u}$ is strictly positive then $\Phi$ is a diffeomorphism of class $C^{m-1}$ and hence $\tau$ has non-vanishing gradient. Now we observe that the CFC-surfaces in $\mathcal{U}$ satisfy the equation
	\begin{equation}
		F(u(\tau, \cdot)) = \tau,
	\end{equation}
	where the left hand-side can be expressed via \eqref{EvHU}. Differentiating both sides with respect to $\tau$, evaluating for $\tau = \tau'$ and taking into account that $u(\tau', \cdot) = 0$ in this coordinate system, we obtain the equation
	\begin{equation}
		-F^{ij}\dot{u}_{ij} + \{F^l_k h_l^m h_m^k + F_k^l \bar{R}_{\alpha\beta\gamma\delta}\nu^\alpha x_l^\beta \nu^\gamma x_m^\delta g^{mk}\} \dot{u} = 1.
	\end{equation}
	Hence in a point, where $\dot{u}$ attains its minimum, we can infer
	\begin{equation}
		\{F^l_k h_l^m h_m^k + F_k^l \bar{R}_{\alpha\beta\gamma\delta}\nu^\alpha x_l^\beta \nu^\gamma x_m^\delta g^{mk}\} \dot{u} \geq 1.
	\end{equation}
	Since the expression in the brackets is always positive, for this fact we refer again to the proof of Lemma \ref{FMonotone}, we conclude that $\dot{u}$ is strictly positive, completing the proof of the Theorem.
\end{proof}

\begin{Bemerkung}
	By looking at tubular neighbourhoods around $M_{c_1}$ and $M_{c_2}$ we obtain new barriers as in the proof of Lemma \ref{LemmaFoliation}. Hence the maximal region which can be foliated by CFC-surfaces of class $C^{m+2, \alpha}$ with positive $F$-curvature is an open subset of $N$ containing $\bar{N}_0$ and the time function in Theorem \ref{FoliationThm} exists on an open interval $I = (a_1, a_2)$ with $a_1 \geq 0$ containing $[c_1, c_2]$.
\end{Bemerkung}

We deliver the Lemma, which has been used in the above Theorem. 
\begin{Lemma}
\label{SuperpositionOp}
Let $K = \bar{\Om}$ be a compact subset of $\bbr^n$, where $\Om$ is open, $E, F$ be Banach spaces, $X \subset E$ an open set and $m \in \bbn$, $m \geq 2$. 
Let $G \in C^m(K\times X, F)$, then the map
\begin{equation}
\begin{split}
\tilde{G}: & C^{k, \beta}(K, X) \rightarrow C^{k,\beta}(K, F),\\
&u(\cdot) \mapsto G(\cdot, u(\cdot)),
\end{split}
\end{equation}
is of class $C^{m-k-1}$, where $k \in \bbn$, $0 \leq k < m$, $0 < \beta \leq 1$.
\end{Lemma}
\begin{proof}
This follows from the proof in \cite[Theorem VII.6.4]{Amann} by using the continuity result from \cite[Theorem 2.1]{ChiapNugari}. 

\end{proof}

Next, we derive some results concerning the area and volume of hypersurfaces between CFC-surfaces.
\begin{Bemerkung}
Now suppose $\mathcal{C} \subset N$ is a cylinder, which can be foliated by CFC-surfaces $M_\tau$, i.e.
\begin{equation}
	\mathcal{C} = \underset{\tau \in J}{\bigcup}M_\tau, \quad J= [c_1, c_2),
\end{equation}
where $0 < c_1 < c_2 \leq b$, then the functions
\begin{align}
	&\tau \mapsto |M_\tau|,\\
	&\tau \mapsto V_{n+1}(M_\tau),
\end{align}
where the functions are defined on $J$, are strictly monotone decreasing and increasing respectively, where the latter follows from the monotonicity of CFC-surfaces. For the former let $M_{\tau_1}$, $M_{\tau_2} \subset \mathcal{C}$ be two CFC-surfaces, $\tau_2>\tau_1$. Then we choose the time-function from Theorem \ref{FoliationThm} and note that in this coordinate system the area is strictly decreasing in view of
\begin{equation}
\label{areaDecreasingEq}
	\frac{d}{dt} \sqrt{\det (\bar{g}_{ij}(t, \cdot))} = - \bar{H}\sqrt{\det (\bar{g}_{ij})} < 0,
\end{equation}
where we used Lemma \ref{FHineq}. Hence the statement. 
\end{Bemerkung}

Now we can derive the following consequence of Theorem \ref{MainTheorem1}
\begin{Proposition}
Let $N$, $F$ be as in Theorem \ref{MainTheorem1} with $m \geq 2$, $0 < \alpha < 1$. Let $M = $ graph $u$ be a compact, spacelike, connected, admissible hypersurface in $N$ of class $C^{4, \alpha}$ satisfying for some $0 < c_1 < c_2 < \infty$
\begin{equation}
	c_1 \leq	F_{|M} \leq c_2, 
\end{equation}
and we assume there exist two CFC-surfaces $M_{c_1}$ and $M_{c_2}$ of class $C^{4,\alpha}$ with $F$-curvature $c_1$ respectively $c_2$.
Then there holds
\begin{equation}
V_{n+1}(M_{c_1}) \leq V_{n+1}(M) \leq V_{n+1}(M_{c_2}),
\end{equation}
and
\begin{equation}
|M_{c_2}| \leq |M| \leq |M_{c_1}|.
\end{equation} 
\end{Proposition}
\begin{proof}
	This follows from Theorem \ref{MainTheorem1} and the remark above.
\end{proof}

In a certain sense we can prove the converse of the above:
\begin{Proposition}
Let $N$, $F$ be as in Theorem \ref{MainTheorem1} with $m\geq 3$, $0 < \alpha < 1$. Let $M_\tau = $ graph $u_\tau$ be a CFC-surface of class $C^{m+2,\alpha}$ with positive $F$-curvature $\tau > 0$. Then $M_\tau$ is the limit hypersurface of a non-trivial curvature flow, which preserves $|M_\tau|$ or $V_{n+1}(M_\tau)$. 
\end{Proposition}
\begin{proof}
We consider a tubular neighbourhood $\mathcal{U} = (-\delta, \delta) \times M_{\tau}$ with $\delta > 0$ around $M_{\tau}$ and work in the corresponding normal gaussian coordinate system of class $C^{m+1, \alpha}$. Furthermore, all hypersurfaces below will be considered as graphs over $M_\tau$, hence $M_\tau = $ graph $0$. For $\delta > 0$ let 
\begin{equation}
	C_\delta := \{u \in C^{m+1,\alpha}(M_\tau): ||u||_{m+1, \alpha, M_\tau} < \delta\},
	\end{equation}
an open subset of $C^{m+1,\alpha}(M_\tau)$. Again we choose $\delta> 0$ small enough, such that, for $u \in C_\delta$, graph $u$ represents a compact, connected, spacelike and admissible hypersurface contained in $\mathcal{U}$. Furthermore we can use $M_\tau$ without loss of generality as the reference hypersurface in the definition of the volume, i.e. $V_{n+1}(M_\tau) = 0$. Now define the following functionals:
\begin{equation}
\begin{split}
	V: & C_\delta \times \bbr \rightarrow \bbr,\\
	&(u, s) \mapsto V_{n+1}(\textnormal{graph } u) - s
\end{split}
\end{equation}
and
\begin{equation}
\begin{split}
	A: &C_\delta \times \bbr \rightarrow \bbr,\\
	&(u, s) \mapsto |\textnormal{graph } u| - s - |\textnormal{graph } 0|.
\end{split}
\end{equation}
Both functionals are continuously differentiable, $V$ as well as $A$ vanish at $(0, 0)$ and they satisfy $D_2V = -1$ and $D_2A = -1$. Hence we can apply the implicit function theorem to obtain an open (and, without loss of generality, connected) neighbourhood $U \subset C_\delta$ of $0$ and a function $\varphi \in C^1(U, \bbr)$, such that
\begin{equation}
	V(u, \varphi(u)) = 0 \quad \forall \, u\in U,
\end{equation} 
respectively
\begin{equation}
	A(u, \varphi(u)) = 0 \quad \forall \, u\in U.
\end{equation}
Since the volume and the area are strictly monotonically increasing and decreasing respectively in the tubular neighbourhood, see \eqref{areaDecreasingEq}, in every arbitrarily small neighbourhood of $0$ in $C_\delta$ there are compact, spacelike, connected, admissible hypersurfaces with bigger and smaller volume respectively area than $M_\tau$, hence $\varphi^{-1}(\bbr_+)$ and $\varphi^{-1}(\bbr_-)$ are both nonempty. 

We conclude that the set $B := \varphi^{-1}(0) - \{0\}$ is nonempty, for otherwise the connected set $\hat{U}:= U - \{0\}$ is identical to $\varphi^{-1}(\bbr_+) \, \dot{\cup} \, \varphi^{-1}(\bbr_-)$, which is a contradiction to the connectedness of $\hat{U}$, since the latter two sets are open in view of the continuity of $\varphi$. 

Hence we obtain a starting hypersurface of class $C^{m+1,\alpha}$, which, when $\delta$ was chosen small enough, fulfills also the barrier requirements, see the proof of Lemma \ref{LemmaFoliation}. Now, since $m\geq 3$, we can apply Theorem \ref{MainTheorem1} to complete the proof.
\end{proof}

\section{Short time existence}
\label{shorttime}
Short time existence for the flow without a global, time-dependent force term is well known, see for example \cite[Chapter 2.5]{GerhCP}. The method employed there is to show first short time existence via the inverse function theorem for a scalar evolution equation (evolution of the graphs) and then using existence results for ordinary differential equations one obtains the desired short time existence for the flow.

We will use a modification of the proof from \cite{GerhCP} and a fixed point argument as in \cite{McCoyMixedArea} to prove the short time existence for the flow with a global force term. 

\begin{Theorem}
	The equation \eqref{floweq} has a solution of class $H^{4+\alpha, \frac{4+\alpha}{2}}(\bar{Q}_\eps)$, where $Q_\eps = [0, \eps) \times M$ and $\eps$ is a small constant.
\end{Theorem}
\begin{proof}
Let $M_0 :=$ graph $u_0$, then we will show the existence of a solution $u\in H^{4+\alpha, \frac{4+\alpha}{2}}( \bar{Q}_\eps)$ to the equation 
\begin{equation}
\label{STf}
\begin{split}
	&\frac{\partial u}{\partial t} + G(x, u, Du, D^2u) + g(x, u, Du) f(t) \equiv \frac{\partial u}{\partial t} + e^{-\psi}v (\Phi - f) = 0,\\
	&u(0) = u_0,
\end{split}
\end{equation}
on a cylinder $Q_\eps := [0, \eps) \times \mathcal{S}_0$ for a small $\eps > 0$. 

$G$ is defined and elliptic for functions $u$ belonging to an open set $\Lambda \subset C^2(\mathcal{S}_0)$, which corresponds to the hypersurfaces being \begin{em}admissible\end{em}:
\begin{equation}
	G^{ij}(x, u, Du, D^2u) < 0.
\end{equation}

	Once the existence for the scalar equation is shown, the arguments in \cite[Chapter 2.5]{GerhCP} can be applied to yield the short-time existence for the parabolic system \eqref{floweq}.

First of all we note, that there exist $\eps_0, \delta > 0$, such that the modified problem
\begin{equation}
\label{STh}
\begin{split}
	&\frac{\partial u}{\partial t} + G(x, u, Du, D^2u) + g(x, u, Du) h(t) = 0,\\
	&u(0) = u_0,
\end{split}
\end{equation}
where we have substituted $f$ by a function $h \in C^{1,\frac{\alpha}{2}}([0, \eps_0])$ and $h$ satisfies 
\begin{equation}
	||h - f(0)||_{1, \frac{\alpha}{2}, \bar{Q}_{\eps_0}} \leq \delta,
\end{equation}
has a unique solution $u \in H^{4+\alpha, \frac{4+\alpha}{2}}(\bar{Q}_{\eps_0})$ with a uniform bound
\begin{equation}
	||u||_{4+\alpha,\frac{4+\alpha}{2}, \bar{Q}_{\eps_0}} \leq c = c(||u_0||_{4+\alpha}, \eps_0, \delta, f(0)).
\end{equation}

In view of the standard parabolic estimates, see \cite[Theorem 2.5.9]{GerhCP}, it is sufficient to show the existence of a solution $u \in H^{2+\beta, \frac{2+\beta}{2}}(\bar{Q}_{\eps_0})$ to \eqref{STh} for some $0 <\beta <\alpha$ and the uniform bound in the corresponding norm. 

The existence and the necessary estimate is shown using the inverse function theorem in a similar manner as in \cite[Chapter 2.5]{GerhCP}, but using the operator
\begin{equation}
	\Psi(u, h) = (\dot{u} + G(x, u, Du, D^2u) + g(x, u, Du)h(t), u(0), h),
\end{equation}
which is well defined in an open subset of $H^{2+\beta, \frac{2+\beta}{2}}(\bar{Q}_\eps) \times C^{1, \frac{\alpha}{2}}([0, \eps])$ 
with image in $(H^{\beta, \frac{\beta}{2}}(\bar{Q}_\eps) \times H^{2+\beta}(\mathcal{S}_0)) \times C^{1, \frac{\alpha}{2}}([0, \eps])$.

We remark that the uniqueness of the solution to the modified problem follows as in the time-independent case by the parabolic maximum principle.\\

To prove short time existence for the problem \eqref{STf}, define the following closed and convex set:
\begin{equation}
	M_{\eps, \delta} := \{h \in C^{1,\frac{\alpha}{2}}([0, \eps]): ||h - f(0)||_{1, \frac{\alpha}{2}} \leq \delta\}.
\end{equation}
	For $h \in M_{\eps, \delta}$, $0< \eps < \eps_0$, $\delta$ as above, denote by $u_h$ a solution to \eqref{STh} and set
\begin{equation}
\begin{split}
T: &M_{\eps, \delta} \rightarrow M_{\eps,\delta},\\
&h \mapsto \frac{\int_{M_t}{\Phi H_k \,\mathrm{d\mu_t}}}{\int_{M_t}{H_k \,\mathrm{d\mu_t}}},
\end{split}
\end{equation}
where the quantities on the right hand side are those belonging to the solution $u_h$ to the problem \eqref{STh}. We will show that in fact $T$ maps $M_{\eps,\delta}$ into itself if $\eps$ is small enough and furthermore $T$ is a compact map, hence maps bounded to precompact sets. The existence of a solution to \eqref{STf} then follows from the Schauder fixed-point theorem.

The essential fact to prove this, is that we have uniform bounds on $u_h$ in $H^{4+\alpha, \frac{4+\alpha}{2}}(\bar{Q}_\eps)$. It follows that $Th$ is uniformly bounded in $C^{1, \frac{1+\alpha}{2}}([0, \eps])$, i.e. there exists a constant $c_0$ independent of $h \in M_{\eps,\delta}$, such that
\begin{equation}
	||Th||_{1, \frac{1+\alpha}{2}, [0,\eps]} \leq c_0.
\end{equation}
 This can be shown by differentiating $Th$ once with respect to $t$, applying partial integration on terms of the form $\dot{u}_{ij}$ and then reminding the definition of the parabolic H\"older spaces. The only critical terms are then of the form $\varphi(x, u, Du, D^2u)^{ijkl} u_{ijk} \dot{u}_l$, hence the claimed H\"older exponential.

The boundedness of $Th$ in $C^{1, \frac{1+\alpha}{2}}([0, \eps])$ implies the compactness of $T$ and by a simple argument it also shows, that $T$ maps $M_{\eps,\delta}$ into itself for small $\eps$:
\begin{equation}
	\big|\frac{d}{dt}Th(t) - \frac{d}{dt}Th(t')\big| \leq  c_0 |t-t'|^{\frac{1+\alpha}{2}} 
	\leq c_0 \eps^{\frac{1}{2}} |t-t'|^{\frac{\alpha}{2}} \leq \delta |t-t'|^{\frac{\alpha}{2}}.
\end{equation}
The $C^0$-norm can be estimated in the same way by noting that $Th(0) = f(0)$ and for the $C^1$-norm we can use the interpolation inequality \eqref{interpolCmalpha}. This completes the proof of the short time existence.
\end{proof}

We conclude this section by showing that the solution is unique, where we remark that this can not be shown as usual by using the maximum principle, in view of the presence of a global term. One rather has to use the idea of uniqueness for weak solutions.

\begin{Proposition}
	The solution to the scalar flow equation is unique in the class $H^{4+\alpha, \frac{4+\alpha}{2}}$.
\end{Proposition}
\begin{proof}
	Let $u, \tilde{u}$ be two solutions of class $H^{4+\alpha, \frac{4+\alpha}{2}}$ in $Q_\eps = [0,\eps) \times \mathcal{S}_0$ of an equation of the form
	\begin{equation}
	\begin{split}
		&\dot{u} + G(x,u,Du, D^2 u) + g(x, u, Du)\frac{\int_{\mathcal{S}_0}{B(x,u,Du, D^2u)\,\mathrm{d\si}}}{\int_{\mathcal{S}_0}{b(x,u,Du, D^2 u)}\,\mathrm{d\si}} = 0,\\
		&u(0) = u_0,
	\end{split}
	\end{equation}
	where for simplicity all functions are supposed to be smooth (the regularity we imposed is also sufficient) and $b > 0$. This equation corresponds to \eqref{totalU} and it suffices to show uniqueness to this equation. Let 
	\begin{equation}
		\varphi := u - \tilde{u}.
	\end{equation} 
	If $\eps$ is sufficiently small, then the convex combination
	\begin{equation}
		u_\tau = \tau u + (1-\tau)\tilde{u}, \quad \tau \in [0,1],
	\end{equation}
	belongs to the open set $\Lambda$, see the above proof for the notation, hence $G$ is well defined for the convex combination. By using the main theorem of calculus we deduce that $\varphi$ satisfies the following equation
	\begin{equation}
		\dot{\varphi} = a^{ij}\varphi_{ij} + b^i \varphi_i + c \varphi + d\int_{\mathcal{S}_0}{\left(\tilde{a}^{ij}\varphi_{ij} + \tilde{b}^i\varphi_i + \tilde{c}\varphi\right) \mathrm{d\si}},
	\end{equation}
	where all the coefficients have bounded derivatives and $a^{ij}$ is uniformly elliptic with ellipticity constant $c_0$. We multiply this equation by $2\varphi$, then we integrate over $\mathcal{S}_0$ and obtain, after using partial integration, the binomial formula and the Schwartz-inequality,
	\begin{equation}
	\begin{split}
		\frac{d}{dt} ||\varphi||_2^2 \equiv \frac{d}{dt} \int_{\mathcal{S}_0}{\varphi^2 \,\mathrm{d\si}} \leq c \,||\varphi||_2^2 + \eps ||D\varphi||_2^2 - c_0 ||D\varphi||_2^2,
	\end{split}
	\end{equation}
	where $\eps >0$ is the constant chosen in the binomial formula yielding the inequality
	\begin{equation}
		a \, b \leq \eps a^2 + \frac{b^2}{4\eps},
	\end{equation} 
	and $c=c(\eps, ||u||_{4+\alpha, \frac{4+\alpha}{2}}, ||\tilde{u}||_{4+\alpha, \frac{4+\alpha}{2}})$.
	Choosing $\eps < c_0$ we deduce that there holds for $h = h(t) = ||\varphi||_2^2$
	\begin{equation}
	\begin{split}
		&\dot h \leq c\, h,\\
		&h(0) = 0.
		\end{split} 
	\end{equation}
	A comparison principle for ordinary differential equations implies $h \equiv 0$. In view of the continuity of $\varphi$, this yields the desired uniqueness.
\end{proof}

\bibliographystyle{plain}
\bibliography{Bibliography}

\end{document}